\documentclass[11pt,letterpaper]{amsart}
\pdfoutput=1

\usepackage{eucal}
\usepackage{graphicx}
\usepackage{color}
\usepackage{amsopn,amssymb}

\definecolor{verydarkblue}{rgb}{0,0,0.4}
\usepackage{hyperref}
\usepackage{mathrsfs}

\usepackage{fancyhdr}

\usepackage{marginnote}

\reversemarginpar

\graphicspath{{figures/}}

\makeatletter%
\def\@commafont{\check@mathfonts
    \fontsize\sf@size\z@\selectfont}
\DeclareRobustCommand{\cb}[1]{{%
\setbox\z@\hbox{#1}%
\ifdim\dp\z@<.1\ht\z@\ooalign{\unhbox\z@\crcr\hidewidth\lower.3ex\hbox{\@commafont,}\hidewidth}%
\else\ooalign{\unhbox\z@\crcr\hidewidth\raise.5ex\hbox{\@commafont`}\hidewidth}%
\fi}}%
\makeatother

\newcommand{\Z}{\mathbb{Z}}
\newcommand{\R}{\mathbb{R}}

\newcommand{\Cc}{\mathcal{C}}

\newcommand{\grad}{\nabla}

\DeclareMathOperator{\im}{Im}
\DeclareMathOperator{\re}{Re}
\DeclareMathOperator{\WP}{WP}

\renewcommand{\Re}{\re}
\renewcommand{\Im}{\im}

\renewcommand{\H}{\mathbb{H}}

\renewcommand{\tilde}[1]{\widetilde{#1}}

\newcommand{\eps}{\varepsilon}

\newcommand{\tec}{Teichm\"uller }
\newcommand{\wep}{Weil-Petersson }

\newcommand{\sT}{\mathcal{T}}

\newcommand{\sM}{\mathcal{M}}

\newcommand{\be}{\begin{equation}}
\newcommand{\ene}{\end{equation}}
\newcommand{\br}{\begin{remark}}
\newcommand{\er}{\end{remark}}
\newcommand{\bl}{\begin{lemma}}
\newcommand{\el}{\end{lemma}}
\newcommand{\bcor}{\begin{cor}}
\newcommand{\ecor}{\end{cor}}
\newcommand{\bpro}{\begin{pro}}
\newcommand{\epro}{\end{pro}}
\newcommand{\ben}{\begin{enumerate}}
\newcommand{\een}{\end{enumerate}}
\newcommand{\bp}{\begin{proof}}
\newcommand{\ep}{\end{proof}}
\newcommand{\bpo}{\begin{pro}}
\newcommand{\epo}{\end{pro}}
\newcommand{\beq}{\begin{equation*}}
\newcommand{\eeq}{\end{equation*}}
\newcommand{\bear}{\begin{eqnarray}}
\newcommand{\eear}{\end{eqnarray}}
\newcommand{\beqar}{\begin{eqnarray*}}
\newcommand{\eeqar}{\end{eqnarray*}}
\newcommand{\brem}{\begin{remark*}}
\newcommand{\erem}{\end{remark*}}
\newcommand{\bt}{\begin{theorem}}
\newcommand{\et}{\end{theorem}}
\newcommand{\bDe}{\begin{Def}}
\newcommand{\eDe}{\end{Def}}

\renewcommand{\leq}{\leqslant}
\renewcommand{\geq}{\geqslant}

\DeclareMathOperator{\length}{Length}
\DeclareMathOperator{\Prob}{Prob}

\DeclareMathOperator{\dArea}{dArea}
\DeclareMathOperator{\diam}{diam}

\DeclareMathOperator{\dist}{dist}
\DeclareMathOperator{\inj}{inj}
\DeclareMathOperator{\Aut}{Aut}
\DeclareMathOperator{\Teich}{Teich}

\DeclareMathOperator{\InR}{InRad}
\DeclareMathOperator{\Diff}{Diff}

\newcommand{\Mod}{\mbox{\rm Mod}}

\newcommand{\lsys}{\ell_{sys}}

\DeclareMathOperator{\Ric}{Ric}
\DeclareMathOperator{\arcsinh}{arcsinh}
\DeclareMathOperator{\HolK}{HolK}

\usepackage{amsthm}

\newcommand{\param}{{\mathchoice{\mkern1mu\mbox{\raise2.2pt\hbox{$\centerdot$}}\mkern1mu}{\mkern1mu\mbox{\raise2.2pt\hbox{$\centerdot$}}\mkern1mu}{\mkern1.5mu\centerdot\mkern1.5mu}{\mkern1.5mu\centerdot\mkern1.5mu}}}

\numberwithin{equation}{section}

\theoremstyle{plain}
\newtheorem{theorem}{Theorem}[section]

\newtheorem{corollary}[theorem]{Corollary}
\newtheorem{lemma}[theorem]{Lemma}
\newtheorem{proposition}[theorem]{Proposition}

\theoremstyle{definition}
\newtheorem{definition}[theorem]{Definition}

\newtheorem*{ques}{Question}

\newtheorem{remark}[theorem]{Remark}
\theoremstyle{definition}
\newtheorem*{rems}{Remarks}
\newtheorem*{rem}{Remark}

\newtheorem*{tms}{Theorem}

\begin{document}

\title[Uniform bounds]
{Systole functions and Weil-Petersson geometry}

\author{Yunhui Wu}
\address{Department of Mathematical Sciences and Yau Mathematical Sciences Center, Tsinghua University, Beijing, China, 100084}

\email{yunhui\_wu@mail.tsinghua.edu.cn}

\begin{abstract}
A basic feature of \tec theory of Riemann surfaces is the interplay of two dimensional hyperbolic geometry, the behavior of geodesic-length functions and \wep geometry. Let $\sT_g$ $(g\geq 2)$ be the \tec space of closed Riemann surfaces of genus $g$. Our goal in this paper is to study the gradients of geodesic-length functions along systolic curves. We show that their $L^p$ $(1\leq p \leq \infty)$-norms at every hyperbolic surface $X\in \sT_g$ are uniformly comparable to $\lsys(X)^{\frac{1}{p}}$ where $\lsys(X)$ is the systole of $X$. As an application, we show that the minimal \wep holomorphic sectional curvature at every hyperbolic surface $X\in \sT_g$ is bounded above by a uniform negative constant independent of $g$, which negatively answers a question of M. Mirzakhani. Some other applications to the geometry of $\sT_g$ will also be discussed.
\end{abstract}


\maketitle

\thispagestyle{fancy}

 \textbf{Notations.} We make the following notations in this paper. 
\ben 
\item For any number $r>0$, we always say $$r^{\frac{1}{\infty}}=1.$$ 

\item We say $$f_1(g)\prec f_2(g) \quad \emph{or} \quad f_2(g)\succ f_1(g)$$ if there exists a universal constant $C>0$, independent of $g$, such that $$f_1(g) \leq C \cdot f_2(g).$$
And we say $$f_1(g) \asymp f_2(g)$$ 
if $f_1(g)\prec f_2(g)$ and $f_2(g)\prec f_1(g)$.

\een

\section{Introduction}
Let $S_g$ be a closed surface of genus $g$ $(g\geq 2)$, and $\sT_g$ be the \tec space of $S_g$. The moduli space $\sM_g$ of $S_g$ is the quotient space $\sT_g/\Mod(S_g)$ where $\Mod(S_g)$ is the mapping class group of $S_g$. For any non-trivial loop $\alpha \subset S_g$ and $X\in \sM_g$, there exists a unique closed geodesic $[\alpha]\subset X$ representing $\alpha$. Denote by $\ell_\alpha(X)$ the length of $[\alpha]$ in $X$. This quantity $\ell_\alpha(\cdot)$ gives a real analytic function on $\sM_g$. The \wep gradient $\grad \ell_{\alpha}(X)$ of $\ell_\alpha(\cdot)$ evaluated at $X$ is a harmonic Beltrami differential of $X$. Actually Gardiner in \cite{Gar75} provided a precise formula for $\grad \ell_{\alpha}(X)$, which one may see in formula \eqref{grad-formula}. Recall that the \emph{systole} $\lsys(X)$ of $X$ is the length of shortest closed geodesics on $X$. We prove
\bt [Uniform $L^p$ bounds] \label{s-ub-lp} For any $p \in [1,\infty]$ and $X\in \sM_g$, then we have 
\ben
\item for any $\alpha \subset X$ with $\ell_{\alpha}(X)=\lsys(X)$, 
\[||\grad \ell_{\alpha}(X)||_p \asymp \lsys(X)^{\frac{1}{p}}.\]

\item For any simple loop $\beta \subset X$ with $\ell_{\beta}(X) \leq L_0$ where $L_0>0$ is any given constant, 
\[||\grad \ell_{\beta}(X)||_p \asymp \ell_\beta(X)^{\frac{1}{p}}.\]
\een
Where $||\grad \ell_{*}(X)||_p:=(\int_X |\grad \ell_{*}(X)|^p \cdot \dArea)^{\frac{1}{p}}.$
\et

\begin{rem}
If $p=2$, in light of \eqref{Rie-lb} which is due to Riera, it suffices to show the uniform upper bounds in Theorem \ref{s-ub-lp}. For this case,
\ben
\item [(i).] Part (1) of Theorem \ref{s-ub-lp} was firstly obtained in \cite[Proposition 2]{W-inradius} whose proof highly relies on formula \eqref{Rie-f} of Riera. In this paper, we give a different and more fundamental proof without using Riera's formula.

\item[(ii).] For Part (2) when $p=2$, the curve $\beta \subset X$ is not required to be a systolic curve of $X$. Theorem \ref{s-ub-lp} was firstly obtained by Wolpert \cite[Lemma 3.16]{Wolpert08} by using Riera's formula; very recently Bridgeman and Bromberg in \cite{BB-20} applied Riera's formula to obtain an upper bound by an explicit elementary function in terms of $\ell_\alpha$ (see \cite[Theorem 1.6]{BB-20}); our proof is different without using Riera's formula.
\een
\end{rem}

In the subsequent subsections, we make several applications of Theorem \ref{s-ub-lp} to the geometry of $\sT_g$, especially on its \wep geometry.
\subsection{Application to \wep curvatures} The \wep curvature of $\sM_g$ has been studied over the past several decades. One may see Section \ref{np} for related references. In this paper we apply Theorem \ref{s-ub-lp} to obtain certain new uniform and optimal bounds on \wep curvatures. The main one is as follows.
\begin{theorem}\label{s-ub-holk}
For any $X \in \sM_g$, let $\alpha \subset X$ be a simple closed geodesic satisfying 
\ben
\item either $\ell_{\alpha}(X)=\lsys(X)$ or

\item $\ell_\alpha(X)\leq L_0$ where $L_0>0$ is any fixed constant,
\een
then the \wep holomorphic sectional curvature
$$\HolK(\grad \ell_{\alpha})(X) \asymp \frac{-1}{\ell_{sys}(X)}.$$
\end{theorem}

\begin{rem}
\ben
\item For the behavior as $\ell_{sys}(X) \to 0$, \cite[Corollary 16]{Wolpert12} or \cite[Theorem 19]{Wolpert12} tells that \emph{$\HolK(\grad \ell_{\alpha})(X)=\frac{-3}{\pi \lsys(X)}+O(\lsys(X))$} where the dominated term $\frac{-3}{\pi \lsys(X)}$ is explicit; however the error term $O(\lsys(X))$ may depend on the topology $g$. The bound in Theorem \ref{s-ub-holk} is uniform.

\item Very recently joint with Martin Bridgeman, we show in \cite{BW-curv} that for any $X \in \sM_g$ with $\lsys(X) \leq 2\arcsinh^{-1}(1)$, then
for any $\mu \neq 0 \in T_X\sM_g^n$, the \wep Ricci curvature $\Ric(\mu)$ satisfies that
\bear\label{ulb-Ric}
\Ric(\mu)\geq -\frac{4}{\lsys(X)}.
\eear 
Theorem \ref{s-ub-holk} tells that the above uniform lower bound in terms of $\lsys(X)$ is optimal. Actually we have
\een
\begin{corollary} \label{s-ub-seck-short}
For any $X \in \sM_g$ with $\lsys(X) \leq 2\arcsinh^{-1}(1)$, then the minimal \wep sectional curvature $K$ at $X$ 
\[\min_{\emph{real plane $\emph{span}\{\mu_1,\mu_2\}  \subset T_X\sM_g$}}K(\mu_1,\mu_2)\asymp \min_{\mu \in T_X\sM_g}\Ric(\mu)\asymp -\frac{1}{\lsys(X)}.\]
\end{corollary}
\bp
Since the \wep sectional curvature is negative \cite{Wolpert86} or \cite{Tromba86}, by \eqref{ulb-Ric} we have
\beqar
-\frac{1}{\lsys(X)} \prec \min_{\mu \in T_X\sM_g}\Ric(\mu) <\min_{\emph{real plane $\emph{span}\{\mu_1,\mu_2\}  \subset T_X\sM_g$}}K(\mu_1,\mu_2).
\eeqar

\noindent Thus, it suffices to show that $\min_{\emph{real plane $\emph{span}\{\mu_1,\mu_2\}  \subset T_X\sM_g$}}K(\mu_1,\mu_2) \prec -\frac{1}{\lsys(X)}$. This can be directly followed by using Theorem \ref{s-ub-holk} because
\beqar
\min_{\emph{real plane $\emph{span}\{\mu_1,\mu_2\}  \subset T_X\sM_g$}}K(\mu_1,\mu_2) &\leq& \min_{\alpha\subset X; \ \ell_\alpha(X)=\lsys(X)}\HolK(\grad \ell_{\alpha})(X) \\
&\asymp& \frac{-1}{\ell_{sys}(X)}
\eeqar
which completes the proof.
\ep
\end{rem}

Combining \cite{Wolf-W-1} with Theorem \ref{s-ub-holk}, we obtain the following uniform upper bounds. 
\bt\label{s-holk-uup}
For any $X \in \sM_g$, 
$$\min_{\mu\in T_{X}\sM_g}\HolK(\mu) \prec -1<0.$$
\et

\begin{rems}
\ben
\item[(i).] If $\lsys(X)$ is big enough, Theorem \ref{s-holk-uup} was firstly obtained by Wolf and the author in \cite{Wolf-W-1} that partially answered a question which is due to Maryam Mirzakhani: \emph{for any fixed constant $\epsilon>0$, whether the \wep curvatures restricting on any $\epsilon$-thick part of $\sM_g$ go to $0$ as the genus $g\to \infty$?} 

\item[(ii).] It was shown in \cite[Theorem 1.4]{W17-wpc} that the limit of the following probability holds: 
$$\lim_{g\to \infty}  \Prob\{X\in \sM_g; \min_{\mu\in T_{X} \sM_g}\HolK(\mu)\prec -1<0\}=1.$$
\een
\noindent Theorem \ref{s-holk-uup} generalizes these two results above and give a complete negative answer to Mirzakhani's question.
\end{rems}

Given any constant $\eps>0$, restricted on the $\eps$-thick part of $\sM_g$ Huang \cite{Huang07} showed that the \wep sectional curvature is uniformly bounded from below by a constant only depending on $\eps$. Later Teo \cite{Teo09} generalized Huang's result to \wep Ricci curvature. The following result roughly says that the uniform lower bounds of Huang and Teo are optimal. More precisely, as an application of Theorem \ref{s-holk-uup} we prove
\begin{corollary} \label{s-ub-seck}
Given any constant $\eps>0$, then for any $X \in \sM_g$ with $\lsys(X)\geq \eps$, the minimal \wep sectional curvature $K$ at $X$
\[\min_{\emph{real plane $\emph{span}\{\mu_1,\mu_2\}  \subset T_X\sM_g$}}K(\mu_1,\mu_2)\asymp \min_{\mu \in T_X\sM_g}\Ric(\mu)\asymp -1.\]
\end{corollary}
\bp
Since $\lsys(X)\geq \eps$, by \cite[Proposition 3.3]{Teo09} of Teo we have
\beqar
-1 \prec \min_{\mu \in T_X\sM_g}\Ric(\mu).
\eeqar

\noindent Since the \wep sectional curvature is negative \cite{Wolpert86} or \cite{Tromba86}, we have
\beqar
\min_{\mu \in T_X\sM_g}\Ric(\mu) <\min_{\emph{real plane $\emph{span}\{\mu_1,\mu_2\}  \subset T_X\sM_g$}}K(\mu_1,\mu_2).
\eeqar

\noindent Thus, it suffices to show that $\min_{\emph{real plane $\emph{span}\{\mu_1,\mu_2\}  \subset T_X\sM_g$}}K(\mu_1,\mu_2) \prec -1 $. This can be directly followed by using Theorem \ref{s-holk-uup} because
\beqar
\min_{\emph{real plane $\emph{span}\{\mu_1,\mu_2\}  \subset T_X\sM_g$}}K(\mu_1,\mu_2)\leq \min_{\mu \in T_{X}\sM_g}\HolK(\mu)\prec -1.
\eeqar
Which completes the proof.
\ep

It was shown in \cite[Corollary 23]{Wolpert12} that there exists a constant $c_{g}<0$, only depending on $g$, such that any subspace $S \subset T_X\sM_g$ with $\dim_{\R}S>(3g-3)$ contains a section with \wep sectional curvature at most $c_g$. In light of Theorem \ref{s-holk-uup}, it would be interesting to know that whether this result of Wolpert is still true if one replaces $c_g$ by a uniform constant $c<0$. More precisely,
\begin{ques}
Does there exist a uniform constant $c<0$, independent of $g$, such that any subspace $S \subset T_X\sM_g$ with $\dim_{\R}S>(3g-3)$ contains a section with \wep sectional curvature at most $c$?
\end{ques}

\subsection{Application to \wep distance}
Let $\Teich(S_g)$ be the \tec space $\sT_g$ endowed with the \wep metric. In \cite{W-inradius} we studied the behavior of the systole function along \wep geodesics and proved the following uniform Lipschitz property.
\begin{tms} \cite[Theorem 1.3]{W-inradius}\label{s-wp-lb}
For all $X,Y\in \Teich(S_{g})$,
\[|\sqrt{\ell_{sys}(X)}-\sqrt{\ell_{sys}(Y)}|\prec \dist_{wp}(X,Y)\]
where $\dist_{wp}$ is the \wep distance.
\end{tms}

\noindent One essential step in the proof of the theorem above is to show Theorem \ref{s-ub-lp} for $p=2$. We outline the proof of the theorem above as follows. One may see \cite{W-inradius} for more details. 
\bp [Outline proof] 
By the real analyticity of the geodesic length function, the Teichm\"uller space $\sT_g$ and the \wep metric, it was shown in \cite[Lemma 4.5]{W-inradius} that the systole function $\lsys(\cdot): \Teich(S_g) \to \mathbb{R}^{>0}$ is piecewisely real analytic in the sense that for any \wep geodesic segment $c:[0,r]\to \Teich(S_g)$ of arc-length parameter, where $r>0$ is any constant, the function $\lsys(c(t)):[0,r]\to R^{>0}$ is piecewisely real analytic.

By \cite{Wolpert87} one may let $c:[0,s]\to \Teich(S_g)$ be the unique \wep geodesic joining $X$ and $Y$ of arc-length parameter where $s=\dist_{wp}(X,Y)>0$. Since the function $\lsys(c(t)):[0,s]\to R^{>0}$ is piecewisely real analytic, one may apply Theorem \ref{s-ub-lp} for the case $p=2$ to obtain that for all $t\in (0,s)$ except finite values,
\beqar \label{s-inr-f-1}
|| \grad \lsys^{\frac{1}{2}}(c(t))||_2\asymp 1.
\eeqar
Then by the Cauchy-Schwarz inequality
\beqar
|\sqrt{\ell_{sys}(X)}-\sqrt{\ell_{sys}(Y)}|&=& |\int_{0}^{s} \langle \grad \lsys^{\frac{1}{2}}(c(t)),c'(t) \rangle_{wp}dt|\\
&\leq &\int_{0}^{s}  ||\grad \lsys^{\frac{1}{2}}(c(t))||_2dt \\
&\asymp & s= \dist_{wp}(X,Y)
\eeqar
where we apply \eqref{s-inr-f-1} in the last equation. Which completes the proof.
\ep
\begin{rems}
\ben
\item Bridgeman and Bromberg in \cite{BB-20} applied their explicit bounds on $|| \grad \lsys(X)||_{wp}$ to show that the function $$\sqrt{\ell_{sys}(\cdot)}:\Teich(S_g) \to \R^{>0}$$ is $\frac{1}{2}$-Lipschitz.

\item In \cite{Wu22}, motivated by the work \cite{RT18} of Rupflin and Topping, we use a purely differential geometrical method to show that $\sqrt{\ell_{sys}(\cdot)}$ is $0.5492$- Lipschitz, without using any estimations on $||\grad(\ell_{sys}(\cdot))||_{wp}$.
\een
\end{rems}

\noindent In \cite{W-inradius} we applied the uniform Lipschitz Theorem above to determine the growth rate of the inradius of $\sM_g$ or $\Teich(S_g)$ for large genus (also for large punctures). Recall that the \emph{inradius} $\InR(\sM_g)$ of $\sM_g$ is defined as $$ \InR(\sM_g):=\sup_{X\in \sM_g}\dist_{wp}(X,\partial \sM_g)$$
where $\partial \sM_g$ is the \wep boundary of $\sM_g$ which is also the same as the boundary of the Deligne-Mumford compactification of $\sM_g$ \cite{Masur76}. It is known by Wolpert \cite[Section 4]{Wolpert08} that for any $X\in \sM_g$, $\dist_{wp}(X,\sM_g^{\alpha})\leq \sqrt{2\pi \cdot \ell_{\alpha}(X)}$ where $\sM_g^{\alpha}$ is the stratum of $\sM_g$ whose pinching curve is $\alpha$. If we choose $\alpha \subset X$ with $\ell_\alpha(X)=\lsys(X)$, then $\dist_{wp}(X, \sM_g^{\alpha})\leq \sqrt{2\pi \lsys(X)}$. It is well-known that $\lsys(X)\prec \ln{(g)}$. So we have $\InR(\sM_g)\prec \sqrt{\ln{(g)}}$. For the other direction, first by Buser-Sarnak \cite{BS94} one may choose a surface $\mathcal{X}_g\in \sM_g$ such that $\lsys(\mathcal{X}_g)\asymp \ln{(g)}$. Clearly the uniform Lipschitz Theorem above implies that $\dist_{wp}(\mathcal{X}_g, \partial \sM_g)\succ \sqrt{\ln{(g)}}$ which in particular implies that $\InR(\sM_g)\succ \sqrt{\ln{(g)}}$. Thus, we have
\begin{tms} \cite[Theorem 1.1]{W-inradius} \label{inradius}
For all $g\geq 2$, $$\InR(\sM_g)\asymp \sqrt{\ln{(g)}}.$$
\end{tms}

\begin{rem}
Bridgeman and Bromberg in \cite{BB-20} showed that
\[\lim \limits_{g\to \infty} \frac{\InR(\sM_g)}{\sqrt{\max\limits_{X\in \sM_g}\lsys(X)}}=\sqrt{2\pi}.\]
One may also see \cite{Wu22} for an alternative proof.
\end{rem}

\begin{rem}
Cavendish-Parlier \cite{CP12} showed that the diameter $\diam(\sM_g)$ of $\sM_g$ satisfies $\sqrt{g} \prec \diam(\sM_{g}) \prec \sqrt{g}\cdot \ln{g}$ where the upper bound refines Brock's quasi-isometry of $\Teich(S_{g})$ to the pants graph \cite{Brock03}. As far as we know, it is still an open problem: \textsl{whether $\sqrt{g}\cdot \ln{g}$ is the correct growth rate for $\diam(\sM_g)$ as $g\to \infty$?}
\end{rem}


\subsection{Application to the $L^p$ $(1< p \leq \infty)$ metric on $\sM_g$}
The \wep metric is the $L^2$-metric on $\sM_g$ and the \tec metric is the $L^1$ metric on $\sM_g$. Similarly for any $p\in [1,\infty]$, the $L^p$ metric on $\sM_g$ can be defined. One may see Section \ref{Lp-metric-inco} for precise definitions. We denote by $(\sM_g, ||\cdot||_{L^p})$ the moduli space $\sM_g$ endowed with the $L^p$ metric, and let $\dist_p(\cdot, \cdot)$ denote the $L^p$-distance on $\sM_g$. As another application of Theorem \ref{s-ub-lp} we show that   
\bt\label{i-Lp-incomplete}
Let $X \in \sM_g$ and $\alpha \subset X$ be a non-trivial simple loop with $\ell_\alpha(X)\leq L_0$ where $L_0>0$ is any given constant. Then for any $p\in (1, \infty]$,
\[\dist_p(X, \sM_g^{\alpha})\prec (\ell_{\alpha}(X))^{1-\frac{1}{p}}\]
where $\sM_g^{\alpha}$ is the stratum of $\sM_g$ whose pinching curve is $\alpha$. In particular, the space $(\sM_g, ||\cdot||_{L^p})$ $(1< p \leq \infty)$ is incomplete.
\et

\begin{rems}
\ben

\item[(i).] As introduced in the proof of Theorem \ref{inradius}, for $p=2$ Wolpert in \cite[Section 4]{Wolpert08} showed that for any $X\in \sM_g$, $$\dist_{wp}(X,\sM_g^{\alpha})\leq \sqrt{2\pi \cdot \ell_{\alpha}(X)}.$$ 
Where $\ell_\alpha(X)$ is not required to be uniformly bounded.

\item[(ii).] Theorem \ref{Lp-incomplete} tells that there exists a constant $K(L_0, p)>0$ depending on $L_0$ and $p$ such that $\dist_p(X, \sM_g^{\alpha})\leq K(L_0,p)\cdot (\ell_{\alpha}(X))^{1-\frac{1}{p}}$ for $X\in \sM_g$ with $\ell_\alpha(X)\leq L_0$. One may easily check in the proof that as $p\to 1$, the constant $K(L_0,p)$ goes to $+\infty$. This can be also followed by the completeness of the \tec metric which is the same as the $L^1$ metric.

\item[(iii).] There are interesting connections between the $L^p$ metric on $\sM_g$ and renormalized volumes of quasi-Fuchsian manifolds. One may see \url{http://math.harvard.edu/~ctm/sem/2014.html} for related discussions. We are grateful to Curt McMullen for noticing us these notes. 
\een
\end{rems}

\subsection*{Plan of the paper.} Section \ref{np} provides some necessary background and the basic properties on two-dimensional hyperbolic geometry and the Weil-Petersson metric. We show that $||\grad \lsys(X)||_\infty$ is uniformly bounded from above in Section \ref{s-uup-elem} and \ref{s-uup-infi}. Theorem \ref{s-ub-lp} is established in Section \ref{s-lp-asymp}. In Section \ref{s-appl-wp} we discuss several applications of Theorem \ref{s-ub-lp} including proving Theorem \ref{s-ub-holk} and \ref{s-holk-uup}. In the last Section \ref{Lp-metric-inco} we prove Theorem \ref{i-Lp-incomplete}.


\subsection*{Acknowledgements.}
The author would like to express his appreciation to Scott Wolpert for many helpful and invaluable conversations on this project. He wants to thank Michael Wolf for a collaboration with him on a related project where some of ideas in this paper emerged from. He also would like to thank Jeffrey Brock, Martin Bridgeman, Shing-Tung Yau and Xuwen Zhu for their interests. This work is supported by the NSFC grant No. $12171263$ and a grant from Tsinghua University.

\section{Notations and Preliminaries}\label{np} 
In this section we set up the notations and provide some necessary background on two-dimensional hyperbolic geometry, \tec theory and the \wep metric.
\subsection{Hyperbolic upper half plane} Let $\mathbb{H}$ be the upper half plane endowed with the hyperbolic metric $\rho(z)|dz|^2$ where
$$\rho(z)=\frac{1}{(\im(z))^2}.$$

A geodesic line in $\mathbb{H}$ is either a vertical line or an upper semi-circle centered at some point on the real axis. For $z=(r,\theta) \in \mathbb{H}$ given in polar coordinates where $\theta \in (0,\pi)$, the hyperbolic distance between $z$ and the imaginary axis $\textbf{i}\mathbb{R^+}$ is
\begin{eqnarray}\label{i-dis}
\dist_{\mathbb{H}}(z, \textbf{i}\mathbb{R^+})=\ln|\csc{\theta}+|\cot{\theta}||.
\end{eqnarray}

\noindent Thus, 
\begin{eqnarray}\label{i-exp}
e^{-2\dist_{\mathbb{H}}(z, \textbf{i}\mathbb{R^+})}\leq \sin^2{\theta}=\frac{\Im^2(z)}{|z|^2}\leq 4e^{-2\dist_{\mathbb{H}}(z, \textbf{i}\mathbb{R^+})}.
\end{eqnarray}

It is known that any eigenfunction with positive eigenvalue of the hyperbolic Laplacian of $\mathbb{H}$ satisfies the mean value property \cite[Coro.1.3]{Fay77}. For $z=(r,\theta) \in \mathbb{H}$ given in polar coordinate, the function 
\[u(\theta)=1-\theta \cot{\theta}\] 
is a positive $2$-eigenfunction. Thus, $u$ satisfies the mean value property. It is not hard to see that $\min\{u(\theta), u(\pi-\theta)\}$ also satisfies the mean value property. Since $\min\{u(\theta), u(\pi-\theta)\}$ is comparable to $\sin^2{\theta}$, from inequality (\ref{i-exp}) we know that the function $e^{-2\dist_{\mathbb{H}}(z, \textbf{i}\mathbb{R^+})}$ satisfies the mean value property in $\mathbb{H}$. The following lemma is the simplest version of \cite[Lemma 2.4]{Wolpert08}.
\begin{lemma}\label{mvp}
For any $r>0$ and $p \in \mathbb{H}$, there exists a positive constant $c(r)$, only depending on $r$, such that
\begin{eqnarray*}
e^{-2\dist_{\mathbb{H}}(p, \textbf{i}\mathbb{R^+})}\leq c(r) \int_{B_{\mathbb{H}}(p;r)}e^{-2\dist_{\mathbb{H}}(z, \textbf{i}\mathbb{R^+})}\dArea
\end{eqnarray*}
where $B_{\mathbb{H}}(p;r)=\{z\in \mathbb{H}; \ \dist_{\mathbb{H}}(p,z)< r\}$ is the hyperbolic geodesic ball of radius $r$ centered at $p$ and $\dArea$ is the hyperbolic area form. 
\end{lemma}
\subsection{Uniform collars} 
Let $X$ be a closed hyperbolic surface of genus $g$ $(g\geq 2)$, and $\alpha \subset X$ be an essential simple loop. There always exists a unique closed geodesic, still denoted by $\alpha$, representing this loop. We denote by $\ell_\alpha(X)$ the length of $\alpha$ in $X$. If the $\delta$-neighborhood $$\Cc_{\delta}(\alpha):=\{x\in X;\ \dist(x,\alpha)<\delta\}$$ is isometric to a cylinder $(\rho,t)\in (-\delta,\delta) \times \mathbb S ^1$, where $\mathbb S ^1 = \R / \Z$, endowed with the standard hyperbolic metric
$$ds^2=d\rho^2 + \ell_{\alpha}(X)^2 \cosh^2\rho dt^2,$$
this set $\Cc_{\delta}(\alpha)$ is called a \emph{$\delta$-collar} of $\alpha$ on $X$. Let $w(\alpha)$ be the supremum of all such $\delta>0$. Then for all $0<\delta<w(\alpha)$, the geodesic arcs of length $\delta$ emanating perpendicularly from $\alpha$ are pairwisely disjoint. We call $\Cc_{w(\alpha)}(\alpha)$ the \emph{maximal collar} of $\alpha$, and $w(\alpha)$ the \emph{width} of $\alpha$.   
First we recall the following classical Collar Lemma. One may refer to \cite[Theorem 4.1.1]{Buser10} or \cite[Theorem 3.8.3]{Hubbard06} for more details.

\begin{lemma}[Collar lemma-1]\label{collar-1}
For any essential simple closed geodesic $\alpha \subset X$, then the width $w(\alpha)$ of $\alpha$ satisfies that
\[w(\alpha)\geq \frac{1}{2} \ln (\frac{\cosh{\frac{\ell_{\alpha}(X)}{2}}+1}{\cosh{\frac{\ell_{\alpha}(X)}{2}}-1}).\]
\end{lemma}
\noindent As the length $\ell_{\alpha}(X)$ of the central closed geodesic $\alpha$ goes to $0$, the width $w(\alpha)$ tends to infinity. Set $f(\ell_\alpha(X))=\frac{1}{2} \ln (\frac{\cosh{\frac{\ell_{\alpha}(X)}{2}}+1}{\cosh{\frac{\ell_{\alpha}(X)}{2}}-1})$. Then the quantity $f(\ell_\alpha(X))$ tends to $0$ as $\ell_{\alpha}(X)$ goes to infinity. Another part of the classical Collar Lemma \cite[Theorem 4.1.1]{Buser10} or \cite[Theorem 3.8.3]{Hubbard06} says that $\Cc_{f(\ell_{\alpha_1}(X))}(\alpha_1)\cap \Cc_{f(\ell_{\alpha_2}(X))}(\alpha_2)=\emptyset$ for any $\alpha_1 \cap \alpha_2 =\emptyset$. In this paper we will not use this disjoint property.

The $\textsl{systole}$ of $X$ is the length of a shortest closed geodesic in $X$. Curves that realize the systole are often referred to \emph{systolic curves}. We denote by $\ell_{sys}(X)$ the systole of $X$. It is known by Buser-Sarnak \cite{BS94} that for each $g \geq 2$, $\lsys(X)\asymp \ln{(g)}$ for certain hyperbolic surface $X$ of genus $g$. Which in particular tells that $\ell_{\alpha}(X)$ could be arbitrarily large for large enough $g$. If $\alpha$ is a systolic curve, the following lemma may be well-known to experts. Since we can not find the obvious references, we provide a proof here for completeness. 
\begin{lemma}[Collar lemma-2]\label{collar-2}
For any $\alpha \subset X$ with $\ell_{\alpha}(X)=\lsys(X)$, then the width $w(\alpha)$ of $\alpha$ satisfies that
\[w(\alpha)\geq \frac{\ell_{\alpha}(X)}{4}.\]
\end{lemma}

\bp
Suppose for contradiction that $w(\alpha)<\frac{\ell_{\alpha}(X)}{4}$. Then there exist two geodesic arcs $c'$ and $c''$ of length $w(\alpha)$ emanating perpendicularly from $\alpha$ such that they have the same endpoint $p_1 \in X$. The two starting points of $c'$ and $c''$ divide $\alpha$ into two arcs. We choose $\alpha_1$ to be the shorter one. In particular, we have $$\length(\alpha_1)\leq \frac{\ell_{\alpha}(X)}{2}.$$ 
For the geodesic triangle $\Delta_1:=c'\cup c''\cup \alpha_1$, we have
\beqar
\length(\Delta_1)&=&2w(\alpha)+\length(\alpha_1)\\
&<& 2\cdot \frac{\ell_{\alpha}(X)}{4}+\frac{\ell_{\alpha}(X)}{2} \\
&=&\ell_{\alpha}(X)=\lsys(X).
\eeqar
Which implies that $\Delta_1$ bounds a disk. On the other hand, the geodesic triangle $\Delta_1=c'\cup c''\cup \alpha_1$ has two interior right angles, in particular, the geodesic triangle $\Delta_1$ has total interior angles no less than $\pi$, which is a contradiction by Gauss-Bonnet.
\ep

The two lemmas above imply that
\begin{proposition}[Uniform Collar]\label{collar-uni}
For any $\alpha \subset X$ with $\ell_{\alpha}(X)=\lsys(X)$, then the width $w(\alpha)$ of $\alpha$ satisfies that
\bear
w(\alpha)\geq \max{\{\frac{1}{2} \ln \frac{\cosh{\frac{\ell_{\alpha}(X)}{2}}+1}{\cosh{\frac{\ell_{\alpha}(X)}{2}}-1}, \frac{\ell_{\alpha}(X)}{4}\}}.
\eear

\noindent Moreover, we have
\[w(\alpha)> \frac{1}{2}.\]
\end{proposition}
\bp
It follows by Lemma \ref{collar-1} and \ref{collar-2} that
\[
w(\alpha)\geq \max{\{\frac{1}{2} \ln \frac{\cosh{\frac{\ell_{\alpha}(X)}{2}}+1}{\cosh{\frac{\ell_{\alpha}(X)}{2}}-1}, \frac{\ell_{\alpha}(X)}{4}\}}.
\]
The minimum of the right hand side holds only if $\frac{1}{2} \ln \frac{\cosh{\frac{\ell_{\alpha}(X)}{2}}+1}{\cosh{\frac{\ell_{\alpha}(X)}{2}}-1}=\frac{\ell_{\alpha}(X)}{4}$. If we set $t=e^{\frac{\ell_{\alpha}(X)}{4}}$, this is equivalent to solve the equation $t^3-t^2-t-1=0.$ A direct computation implies that the unique real solution is $t\sim 1.8392$. So we have
$$\max{\{\frac{1}{2} \ln \frac{\cosh{\frac{\ell_{\alpha}(X)}{2}}+1}{\cosh{\frac{\ell_{\alpha}(X)}{2}}-1}, \frac{\ell_{\alpha}(X)}{4}\}}\sim \ln {(1.8392)}\sim 0.5877>\frac{1}{2}.$$ 

The proof is complete.  
\ep

\subsection{\tec space}
Let $S_{g}$ be a closed surface of genus $g \ (g\geq 2)$. Let $\textsl{M}_{-1}$ be the space of Riemannian metrics on $S_{g}$ with constant curvature $-1$, and $X=(S_{g},\sigma(z)|dz|^2) \in \textsl{M}_{-1}$. The group $\Diff_+(S_g)$, which is the group of orientation-preserving diffeomorphisms of $S_g$, acts by pull back on $\textsl{M}_{-1}$. In particular this also holds for the normal subgroup $\Diff_0(S_g)$ of $\Diff_+(S_g)$, the group of diffeomorphisms 
isotopic to the identity. The group $\Mod(S_g):=\Diff_+(S_g)/\Diff_0(S_g)$ is called the \textsl{mapping class group} of $S_{g}$. 

The \textsl{Teichm\"uller space} $\sT_{g}$ of $S_{g}$ is defined as
\begin{equation}
\nonumber \sT_{g}:=M_{-1}/\Diff_0(S_g).  
\end{equation}

The \textsl{moduli space} $\sM_{g}$ of $S_{g}$ is defined as
\begin{equation}
\nonumber \sM_{g}:=\sT_{g}/\Mod(S_g).  
\end{equation}

The \tec space $\sT_{g}$ is a real analytic manifold. Let $\alpha \subset S_{g}$ be an essential simple closed curve, then for any $X \in \sT_{g}$, there exists a unique closed geodesic, still denoted by $\alpha$ in $X$, which represents for $\alpha$ in the fundamental group of $S_{g}$. The geodesic length function $\ell_{\alpha}(\cdot) $ defines a function on $\sT_{g}$. It is well-known that this function $\ell_{\alpha}(\cdot)$ is real-analytic on $\sT_{g}$. For more details on \tec theory, one may refer to \cite{Buser10, IT92, Hubbard06}.

Recall that the systole $\ell_{sys}(X)$ of $X\in \sT_g$ is the length of shortest closed geodesics on $X$. It defines a continuous function $\lsys(\cdot): \sT_{g}\to \mathbb{R}^+$, which is called the \textsl{systole function} on $\sT_g$. In general, the systole function is continuous but not smooth because of corners where there may exist multiple essential simple closed geodesics realizing the systole. This function is very useful in \tec theory. One may refer to \cite{Akr03, Gen15, Sch93} for more details.

\subsection{\wep metric}
Let $X=(S_g,\sigma(z)|dz|^2)\in \sT_g$. The tangent space $T_X\sT_g$ at $X$ is identified with the space of {\it harmonic Beltrami differentials} on $X$, i.e. forms on $X$ expressible as 
$\mu=\frac{\overline{\psi}}{\sigma}$ where $\psi \in H^0(X, K^2)$ is a holomorphic quadratic differential on $X$. Let $z=x+\textbf{i}y$ and $\dArea=\sigma(z)dxdy$ be the volume form. The \textit{Weil-Petersson metric} is the Hermitian
metric on $\sT_g$ arising from the the \textit{Petersson scalar  product}
\begin{equation}
 \left<\varphi,\psi\right>_{wp}= \int_X \frac{\varphi \cdot \overline{\psi}}{\sigma^2}\dArea\nonumber
\end{equation}
via duality. We will concern ourselves primarily with its Riemannian part $ds^2_{\WP}$. Throughout this paper we denote the Teichm\"uller space endowed with the Weil-Petersson metric by $\Teich(S_g)$. By definition it is easy to see that the mapping class group $\Mod(S_g)$ acts on $\Teich(S_g)$ as isometries. Thus, the \wep metric descends into a metric, also called the \wep metric, on the moduli space $\sM_g$. Throughout this paper we also denote by $\sM_g$ the moduli space endowed with the Weil-Petersson metric.

\subsection{Fenchel-Nielsen deformation}
For any essential simple closed curve $\alpha\subset S_{g}$, the geodesic length function $\ell_{\alpha}(\cdot)$ is real-analytic over $\sT_g$. Let $X= (S_g, \sigma(z)|dz|^2) \in \sM_g$ be a hyperbolic surface and $\Gamma$ be its associated Fuchsian group. Recall that we also let $\alpha$ denote its unique closed geodesic representative in $X$. One may denote by $A:z \to e^{\ell_\alpha(X)}\cdot z$ the deck transformation on the upper half plane $\mathbb{H}$ corresponding to the simple closed geodesic $\alpha \subset X$.
 
\begin{definition}
Associated to the geodesic $\alpha$, we define the following holomorphic quadratic differential on $X$ as
\begin{equation}
\Theta_{\alpha}(z):=\displaystyle \sum_{E\in \left<A\right>\backslash \Gamma} \frac{E'(z)^2}{E(z)^2}dz^2
\end{equation}
where $\left<A\right>$ is the cyclic group generated by $A$. 
\end{definition}

Let $\grad \ell_{\alpha}(X)\in T_X \sM_g$ be the \wep gradient of the geodesic length function $\ell_{\alpha}(\cdot)$ at $X$. By \cite{Wolpert82} it is known that $\grad \ell_{\alpha}=-2\textbf{i}\cdot t_\alpha$ where $t_{\alpha}$ is the infinitesimal Fenchel-Nielsen right twist deformation along $\alpha$. Moreover,  
\bear \label{grad-formula}
\ \ \ \ \grad \ell_{\alpha}(X)(z)=\frac{2}{\pi}\frac{\overline{\Theta}_{\alpha}(z)}{\rho(z)|dz|^2}= \frac{2}{\pi}\sum_{E\in \left<A\right>\backslash \Gamma} \frac{\overline{E}'(z)^2}{\overline{E}(z)^2 \rho(z)}\frac{d\overline{z}}{dz}\in T_X \sM_g
\eear
where $\rho(z)|dz|^2=\frac{|dz|^2}{(\Im(z))^2}$ is the hyperbolic metric on the upper half plane.

A special formula of Riera in \cite{Rie05} says that
\begin{equation}\label{Rie-f}
\langle\,\grad \ell_{\alpha},\grad \ell_{\alpha}\rangle_{wp}(X)=\frac{2}{\pi}\left(\ell_{\alpha}(X)+\displaystyle \sum_{E\in \{\left<A\right>\backslash\Gamma/\left<A\right>-id\}} (u \ln{\frac{u+1}{u-1}}-2)\right)
\end{equation}
where $u=\cosh{(\dist_{\mathbb{H}}(\tilde{\alpha}, E\circ \tilde{\alpha}))}$ and the double-coset of the identity element is omitted from the sum. In particular, one has
\begin{equation}\label{Rie-lb}
\langle\,\grad \ell_{\alpha},\grad \ell_{\alpha}\rangle_{wp}(X)>\frac{2}{\pi}\cdot \ell_{\alpha}(X).
\end{equation}

\begin{definition}
For any $p\in [1,\infty]$, $X \in \sM_g$ and $\mu\in T_X\sM_g$, we define the \emph{$L^p$-norm} $||\mu||_p$ of $\mu$ to be
\bear \label{def-mu-p}
||\mu||_p:=\left(\int_X |\mu|^p \cdot \dArea \right)^{\frac{1}{p}}.
\eear
\end{definition}

\noindent If $p=2$, this is the standard \wep norm. If $p=\infty$, $||\mu||_\infty=\max_{z \in X}|\mu(z)|$. One main goal of this paper is to study $||\grad \ell_{\alpha}(X)||_p$ when $\alpha \subset X$ is a systolic curve.

\subsection{Weil-Petersson curvatures.} The curvature tensor of the \wep metric on $\sM_g$ is given as follows. Let $\mu_{i},\mu_{j}$ be two elements in the tangent space $T_X\sM_g$ at $X$, so that the metric tensor might be written in local coordinates as
\begin{eqnarray*}
g_{i \overline{j}}=\int_X \mu_{i} \cdot  \overline{\mu_j} \dArea.
\end{eqnarray*} 

For the inverse of $(g_{i\overline{j}})$, we use the convention
\begin{eqnarray*}
g^{i\overline{j}} g_{k\overline{j}}=\delta_{ik}.
\end{eqnarray*}

Then the curvature tensor is given by
\begin{eqnarray*}
R_{i\overline{j}k\overline{l}}=\frac{\partial^2}{\partial t^{k}\partial \overline{t^{l}}}g_{i\overline{j}}-g^{s\overline{t}}\frac{\partial}{\partial t^{k}}g_{i\overline{t}}\frac{\partial}{\partial \overline{t^{l}}}g_{s\overline{j}}.
\end{eqnarray*}

The following curvature formula was established in \cite{Tromba86, Wolpert86}. It has been applied to study various curvature properties of the Weil-Petersson metric. Wolpert \cite{Wolpert86} and Tromba \cite{Tromba86} independently showed that $\sM_g$ has negative sectional curvature. In \cite{Schu86} Schumacher showed that $\Teich(S_g)$ has strongly negative curvature in the sense of Siu. Liu-Sun-Yau in \cite{LSY08} showed that $\sM_g$ has dual Nakano negative curvature, which says that the complex curvature operator on the dual tangent bundle is positive in some sense. It was shown in \cite{W14-co} that the $\sM_g$ has non-positive definite Riemannian curvature operator. One can also see \cite{Huang05, Huang07, Huang07-a, LSY04, MZhu17-wp, Teo09, Wolpert11, Wolpert12, W17-wpc, Wolf-W-1} for other aspects of the \wep curvatures of $\sM_g$.

Set $D=-2(\Delta-2)^{-1}$ where $\Delta$ is the Beltrami-Laplace operator on $X=(S_g,\sigma(z)|dz|^2) \in \sM_g$. The operator $D$ is positive and self-adjoint. 
\begin{theorem}[Tromba, Wolpert]\label{cfow} 
The curvature tensor satisfies
\[R_{i\overline{j}k\overline{l}}=\int_{X} D(\mu_{i}\mu_{\overline{j}})\cdot (\mu_{k}\mu_{\overline{l}}) \dArea+\int_{X} D(\mu_{i}\mu_{\overline{l}})\cdot (\mu_{k}\mu_{\overline{j}}) \dArea.\]
\end{theorem}
\subsubsection{Weil-Petersson holomorphic sectional curvatures.}\label{subsec:hol-sec-curv}
Recall that a holomorphic sectional curvature is a sectional curvature along a holomorphic line. Let $\mu \in T_X\sM_g$. Then Theorem \ref{cfow} tells that the \wep holomorphic sectional curvature $\HolK(\mu)$ along the holomorphic line spanned by $\mu$ is
\[\HolK(\mu)=\frac{-2\cdot \int_{X} D(|\mu|^2)\cdot (|\mu|^2) \dArea}{||\mu||_{\WP}^4}.\]
Assume that $||\mu||_{\WP}=1$. By the Cauchy-Schwarz inequality and an estimation of Wolf in \cite{Wolf12} we know that \cite[Proposition 2.7]{Wolf-W-1} 
\begin{equation}\label{HolK-eq}
-2 \int_{X} |\mu|^4 \dArea \leq \HolK(\mu)\leq -\frac{2}{3} \int_{X} |\mu|^4 \dArea.
\end{equation}

\subsubsection{Weil-Petersson sectional curvatures.}\label{subsec:sec-curv}
Let $\mu_i, \mu_j \in T_X\sM_g$ be two orthogonal tangent vectors with $||\mu_i||_{\WP}=||\mu_j||_{\WP}=1$. Then the sectional curvature $K(\mu_i,\mu_j)$ of the plane spanned by the real vectors corresponding to $\mu_i$ and $\mu_j$ is \cite{Wolpert86}
\beqar
K(\mu_i,\mu_j)&=&\Re{\int_X D(\mu_i \mu_{\overline{j}})\mu_i \mu_{\overline{j}}\dArea}-\frac{1}{2}\int_X D(\mu_i \mu_{\overline{j}})\mu_{\overline{i}}\mu_j\dArea \\
&&-\frac{1}{2}\int_X D(|\mu_i|^2)|\mu_j|^2\dArea.
\eeqar
Apply \cite[Lemma 4.3]{Wolpert86} and Cauchy-Schwarz inequality one may have
\beqar
\int_X D(|\mu_i||\mu_j|)|\mu_i||\mu_j|\dArea &\leq& \int_X D(|\mu_i|^2)^{\frac{1}{2}} D(|\mu_j|^2)^{\frac{1}{2}}|\mu_i||\mu_j|\dArea\\
&\leq & (\int_X D(|\mu_i|^2)|\mu_j|^2\dArea)^{\frac{1}{2}}\\
&& \times (\int_X D(|\mu_j|^2)|\mu_i|^2\dArea)^{\frac{1}{2}} \\
&=& \int_X D(|\mu_i|^2)|\mu_j|^2\dArea.
\eeqar
Where in the last inequality we apply the fact that the operator $D$ is self-adjoint. It is clear that $$|\int_X D(\mu_i \mu_{\overline{j}})\mu_i \mu_{\overline{j}}\dArea|\leq \int_X D(|\mu_i||\mu_j|)|\mu_i||\mu_j|\dArea$$ and $$|\int_X D(\mu_i \mu_{\overline{j}}) \mu_{\overline{i}}\mu_j\dArea| \leq \int_X D(|\mu_i||\mu_j|)|\mu_i||\mu_j|\dArea.$$ Then one may have the following bound,
\begin{equation}\label{SecK-eq}
K(\mu_i, \mu_j)\geq -2\int_X D(|\mu_i|^2)|\mu_j|^2\dArea.
\end{equation}

\subsubsection{Weil-Petersson Ricci curvatures.}\label{subsec:ric-curv}
Let $\{\mu_{i}\}_{i=1}^{3g-3}$ be a holomorphic orthonormal basis of $T_X\sM_g$. Then the Ricci curvature $\Ric(\mu_i)$ of $\sM_g$ at $X$ in the direction $\mu_{i}$ is given by
\begin{eqnarray*}
&&\Ric(\mu_{i})=-\sum_{j=1}^{3g-3}R_{i\overline{j}j\overline{i}}\\
&=&-\sum_{j=1}^{3g-3}(\int_{X} D(\mu_{i}\mu_{\overline{j}})\cdot (\mu_{j}\mu_{\overline{i}}) \dArea+\int_{X} D(|\mu_{i}|^2)\cdot (|\mu_{j}|^2) \dArea).
\end{eqnarray*}
Since $\int_X D(f)\cdot \overline{f} \dArea \geq 0$ for any function $f$ on $X$, by applying the argument in the proof of \eqref{SecK-eq} one may have
\begin{equation}\label{RicK-eq}
-2\leq \frac{\Ric(\mu_i)}{\sum_{j=1}^{3g-3}\int_{X} D(|\mu_{i}|^2)\cdot (|\mu_{j}|^2)\dArea} \leq -1.
\end{equation}


\section{A upper bound for $\grad \ell_\alpha(X)$}\label{s-uup-elem}
In this section we provide an elementary upper bound $H(z)$ for $|\grad \ell_\alpha(X)(z)|$ and detect the region where $H(z)$ attains its maximal value.

Let $\alpha \subset X\in \sM_g$ be an essential non-trivial loop. Up to conjugacy, one may assume that the closed geodesic, still denoted by $\alpha$, representing $\alpha$ corresponds to the deck transformation $A: z\to e^{\ell_{\alpha}(X)}\cdot z$ with axis $\tilde{\alpha}=\textbf{i}\mathbb{R}^+$ which is the imaginary axis, and the fundamental domain $\mathcal{A}_\alpha$ with respect to this cyclic group $\left<A\right>$ is 
\bear
\mathcal{A}_\alpha=\{z\in \mathbb{H}; 1\leq |z|\leq e^{\ell_{\alpha}(X)}\}.
\eear

Recall that the \wep gradient of $\ell_\alpha(\cdot)$ at $X$ is \cite{Gar75, Wolpert82} $$\grad \ell_{\alpha}(X)(z)=\frac{2}{\pi}\sum_{E\in \left<A\right>\backslash \Gamma} \frac{\overline{E}'(z)^2}{\overline{E}(z)^2 \rho(z)}\frac{d\overline{z}}{dz} \in T_X \sM_g$$ where $\rho(z)=\frac{1}{\Im(z)^2}$. Since $\rho(\gamma (z)) |\gamma'(z)|^2=\rho(z)$ for any $\gamma \in \Aut(\mathbb{H})$, the triangle inequality gives that for all $z\in \mathcal{A}_\alpha$,
\begin{equation}\label{1}
|\grad \ell_{\alpha}(X)|(z)\leq \frac{2}{\pi} \sum_{E\in \left<A\right>\backslash\Gamma} \frac{1}{|E(z)|^2}\times \frac{1}{\rho(E(z))}.
\end{equation}

Set 
\bear \label{H-expr}
H(z):&=&\sum_{E\in \left<A\right>\backslash \Gamma} \frac{1}{|E(z)|^2}\times \frac{1}{\rho(E(z))}\\
&=&\sum_{E\in \left<A\right>\backslash \Gamma} \sin^2 (\theta(E(z))) \nonumber
\eear
where we write $E(z)=(r(E(z)),\theta(E(z)))$ in the polar coordinate of $\H$. Since $H(\gamma \cdot z)=H(z)$ for any $\gamma \in \Gamma$, it descends into a function on $X$ $$h:X \to \R^{>0}$$ defined as $$h(\mathcal{\pi}(z)):=H(z)$$ where $\mathcal{\pi}:\H \to X$ is the covering map and $z \in \mathbb{H}$ is any lift point of $\mathcal{\pi}(z)\in X$.

By \eqref{1} we know that
\be \label{2-3-0-0}
|\grad \ell_{\alpha}(X)|(\mathcal{\pi}(z)) \leq \frac{2}{\pi}\cdot H(z).
\ene

Now we estimate $\max\limits_{z\in \H} H(z)$ or $\max\limits_{z\in X}h(z)$. First we compute the Laplacian of $H(z)$ to detect the rough region where $\max_{z\in \H} H(z)$ holds.

We rewrite $\frac{1}{|E(z)|^2}\times \frac{1}{\rho(E(z))}$ as
\[f(z)=\frac{1}{4}\times \frac{(E(z)-\overline{E}(z))(\overline{E}(z)-E(z))}{E(z)\times \overline{E}(z)}.\]

\noindent Note that $E(z)$ is analytic, i.e., $E_{\overline{z}}(z)=0$. So we have
\begin{eqnarray}
\frac{\partial}{\partial z}f(z)&=&\frac{1}{4}\times ( \frac{2E_z(z)\cdot (\overline{E}(z)-E(z))}{E(z)\times \overline{E}(z)}\\
&-&\frac{(E(z)-\overline{E}(z))\cdot (\overline{E}(z)-E(z))}{E^2(z)\times \overline{E}(z)}\cdot E_z(z) ) \nonumber\\
&=&\frac{1}{4}\frac{(E(z)+\overline{E}(z))\cdot (\overline{E}(z)-E(z))}{E^2(z)\times \overline{E}(z)}\cdot E_z(z).\nonumber
\end{eqnarray}

\noindent Take one more derivative we get,
\begin{eqnarray}\label{2}
4\frac{\partial^2}{\partial \overline{z}\partial z}f(z)&=& \frac{\overline{E}(z)-E(z)}{E^2(z)\times \overline{E}(z)}\cdot E_z(z)\cdot  \overline{E}_{\overline{z}}(z) \\
&+& \frac{E(z)+\overline{E}(z)}{E^2(z)\times \overline{E}(z)}\cdot E_z(z)\cdot  \overline{E}_{\overline{z}}(z) \nonumber  \\ \nonumber
&-&\frac{(E(z)+\overline{E}(z))\cdot (\overline{E}(z)-E(z))}{E^2(z)\times \overline{E}^2(z)}\cdot E_z(z)\cdot \overline{E}_{\overline{z}}(z) \\
&=&\frac{|E_z(z)|^2}{|E(z)|^4}\times ( E^2(z)+\overline{E}^2(z)) .\nonumber
\end{eqnarray}

\noindent Let $E(z)=\Re{E(z)}+\textbf{i}\cdot \Im{E(z)}$. Then 
\be \label{2-3-0}
\frac{\partial^2}{\partial \overline{z}\partial z}f(z)=\frac{|E_z(z)|^2}{2|E(z)|^4}\times \left((\Re{E})^2-(\Im{E})^2\right). 
\ene

\noindent Recall that for $z=(r,\theta) \in \mathbb{H}$ given in the polar coordinate where $\theta \in (0,\pi)$, the hyperbolic distance between $z$ and the imaginary axis $\textbf{i}\mathbb{R^+}$ is given in \eqref{i-dis} saying
\begin{eqnarray*}
\dist_{\mathbb{H}}(z, \textbf{i}\mathbb{R^+})=\ln|\csc{\theta}+|\cot{\theta}||.
\end{eqnarray*}

\noindent Consider the set $$\mathcal{P}_\alpha=\{(x,y)\in \mathbb{H}; \ y\geq |x|\} \cap \{(r,\theta); \ 1\leq r \leq \ell_{\alpha}(X)\}.$$ Geometrically, $\mathcal{P}_\alpha$ is the set in $\H$ which has distance to the imaginary axis with
\[\dist_{\mathbb{H}}(z, \textbf{i}\mathbb{R^+})\leq \ln(\sqrt{2}+1).\]

\noindent Equation (\ref{2-3-0}) tells that if $E(z) \notin \mathcal{P}_\alpha$, then we have
\bear
\Delta f(z) \geq 0.
\eear

Recall that $\alpha$ is the closed geodesic representing for $\alpha$ in $X$ which is lifted into $\H$ as $\{\textbf{i}\cdot t\}$ where $1\leq t \leq e^{\ell_\alpha(X)}$. For any $\mathcal{\pi}(z) \in X$ with $\dist(\pi(z),\alpha)>\ln(\sqrt{2}+1)$, we have that for any lift $z \in \mathbb{H}$ of $\mathcal{\pi}(z) \in X$ and any $E \in \Gamma$, 
\[\dist_{\mathbb{H}}(E(z), \textbf{i}\mathbb{R^+})> \ln(\sqrt{2}+1).\]
Therefore equation (\ref{2-3-0}) implies that  for any $z \in \mathbb{H}$ with $\dist(\mathcal{\pi}(z), \alpha)> \ln(\sqrt{2}+1)$ we have
\begin{eqnarray}\label{2-1}
\Delta H(z)=\sum_{E\in \left<A\right>\backslash \Gamma} \frac{|E_z(z)|^2}{4|E(z)|^4}\times ( E^2(z)+\overline{E}^2(z))>0.
\end{eqnarray}

\noindent Recall that $h(\mathcal{\pi}(z))=H(z)$ for any $z \in \H$. The following proposition tells that the maximum of $h$ can only happen near the central closed geodesic. More precisely,

\begin{proposition}\label{p-1} We have
\[\max_{\pi(z) \in X} h(\pi(z))=\max_{\pi(z) \in X;\ \dist(\pi(z), \alpha)\leq \ln(\sqrt{2}+1)} h(\pi(z)).\]
\end{proposition}
\bp
Suppose for contradiction that there exists a point $\pi(w) \in X$ with $\dist(\pi(w), \alpha)>\ln(\sqrt{2}+1)$ such that
\[h(\pi(w))=\max_{\pi(z) \in X} h(\pi(z)).\]
Let $w\in \mathbb{H}$ be a lift of $\pi(w) \in X$. Then we have
\[H(w)=\max_{z \in \mathbb{H}}H(z).\]
In particular, by the maximal principal we have
\bear
\Delta H(w) \leq 0.
\eear
On the other hand, since $\dist(\pi(w), \alpha)>\ln(\sqrt{2}+1)$, for any $E\in \left<A\right>\backslash\Gamma$ we have
\[\dist_{\mathbb{H}}(E\circ w ,\alpha)>\ln(\sqrt{2}+1).\]
Then it follows by \eqref{2-1} that
\bear
\Delta H(w) > 0
\eear
which is a contradiction. The proof is complete.
\ep

\br
The proposition above tells that the maximum of $H$ can only happen in the $ \ln(\sqrt{2}+1)-$neighborhood of $\cup_{E\in \left<A\right>/\Gamma}E\circ \textbf{i}\mathbb{R^+}$, which is also the same as $\cup_{E\in \left<A\right>/\Gamma} E\circ \mathcal{P}_\alpha$. 
\er

\section{Uniform upper bounds for $L^\infty$-norms}\label{s-uup-infi}
In this section we will prove the uniform upper bounds in Theorem \ref{s-ub-lp} for $p=\infty$. Part (1) and (2) of Theorem \ref{s-ub-lp} in this case will be proved separately. We first show  
\begin{proposition}\label{uub-infi}
Let $X\in \sM_g$ and $\alpha \subset X$ with $\ell_{\alpha}(X)=\lsys(X)$. Then 
\bear \label{uub-lin}
||\grad \ell_{\alpha}(X)||_\infty \prec 1. \nonumber
\eear
\end{proposition}

The proof of Proposition \ref{uub-lin} is splitted into two parts. We first deal with the case that $X$ is always contained in certain fixed thick part of the moduli space $\sM_g$.

\bl \label{uub-large}
For any given constant $\eps_0>0$ and $\alpha \subset X \in \sM_g$ with $\ell_{\alpha}(X)=\lsys(X) \geq \eps_0$, then there exists a uniform constant $C_1(\eps_0)>0$, only depending on $\eps_0$, such that
\[||\grad \ell_{\alpha}(X)||_\infty\leq C_1(\eps_0).\]
\el

\bp
By \eqref{2-3-0-0} it suffices to show that \[\max_{z \in \mathbb{H}}H(z)\leq C'_1(\eps_0)\]
for some uniform constant $C'_1(\eps_0)>0$ only depending on $\eps_0$.

Let $z_0\in \mathcal{F}_\alpha \subset \mathcal{A}_{\alpha}$, where $\mathcal{F}_\alpha$ is the fundamental domain of $X$ which contains the lift $\{\textbf{i}\cdot t\}_{1\leq t \leq e^{\ell_\alpha(X)}}$ of the shortest closed geodesic $\alpha$ in $X$, such that $H$ attains its maximum at $z_0$. That is,
\[H(z_0)=\max_{z \in \mathbb{H}}H(z).\]

\noindent Since $\lsys(X)\geq \eps_0$, it follows by the triangle inequality that for any $\gamma_1\neq \gamma_2 \in \Gamma$, the geodesic balls satisfy that
\bear \label{ball-disj-1}
\gamma_1\circ B(z_0; \frac{\eps_0}{8}) \cap \gamma_2\circ B(z_0; \frac{\eps_0}{8})=\emptyset.
\eear
Recall that $\frac{1}{|E(z)|^2}\times \frac{1}{\rho(E(z))}=\sin^2(\theta(E(z)))$ where we use polar coordinate $(r,\theta)$ for $E(z)$. For any $E \notin \left<A\right>$, the point $E\circ z_0$ must lie outside of a lift $\tilde{\Cc}_{w(\alpha)}(\alpha)$ of the maximal collar $\mathcal{C}_{w(\alpha)}(\alpha)$ where
\[\tilde{\Cc}_{w(\alpha)}(\alpha)=\{z\in \mathcal{A}_\alpha; \ \dist_{H}(z, \textbf{i}\cdot \R^{+})\leq w(\alpha)\}.\]

\noindent In particular, by Lemma \ref{collar-2} we have for all $E\notin \left<A\right>$,
\beqar 
\dist_{\H}(E\circ z_0, \textbf{i}\cdot \mathbb{R^+})\geq \frac{\lsys(X)}{4}.
\eeqar

\noindent Then it follows by the triangle inequality that for any $E \notin \left<A\right>$ and $z\in E\circ B_{\mathbb{H}}(z_0;\frac{\eps_0}{8})$, 
\bear \label{larg-1-1}
\dist_{\H}(z, \textbf{i}\cdot \mathbb{R^+})&\geq&\dist_{\H}(E\circ z_0, \textbf{i}\cdot \mathbb{R^+})-\dist_\H(E\circ z_0, z)  \\
&\geq& \frac{\lsys(X)}{4}-\frac{\eps_0}{8}\nonumber \\
&\geq& \frac{\lsys(X)}{8}.\nonumber
\eear

\noindent Then by formula \eqref{i-dis} we have that for any $E \notin \left<A\right>$ and $z\in E\circ B_{\mathbb{H}}(z_0;\frac{\eps_0}{8})$, 
\beqar \label{larg-1-2-0}
\ln {(\frac{2}{\sin (\theta(z))})} \geq \frac{\lsys(X)}{8}.
\eeqar
Recall that $e^x\geq x$ for all $x\geq 0$. Thus, for all $z\in E\circ B_{\mathbb{H}}(z_0;\frac{\eps_0}{8})$ where $E \notin \left<A\right>$ we have 
\bear \label{larg-1-2-0}
\sin {(\theta(z))}\leq \frac{16}{\lsys(X)}.
\eear

\noindent Now we apply the mean value inequality. Recall that \eqref{H-expr} says that
\[H(z)=\sum_{E\in \left<A\right>\backslash \Gamma} \sin^2 (\theta(E(z))).\]

\noindent Then it follows by \eqref{i-exp} and Lemma \ref{mvp} that
\bear \label{lar-1-3}
&& H(z_0)-\sin^2(\theta(z_0)) \leq \sum_{E \neq \left<A\right>\in \left<A\right>\backslash\Gamma} 4 e^{-2\dist_{\mathbb{H}}(E\circ z_0, \textbf{i}\cdot \mathbb{R^+})}  \\
&&\leq   4c(\frac{\eps_0}{8})\sum_{E\neq \left<A\right> \in \left<A\right>\backslash\Gamma} \int_{E\circ B_{\mathbb{H}}(z_0;\frac{\eps_0}{8})}e^{-2\dist_{\mathbb{H}}(z, \textbf{i}\cdot \mathbb{R^+})}\dArea. \nonumber
\eear

\noindent Then by the triangle inequality and \eqref{larg-1-2-0} we know that for all $E\notin \left<A\right>$,
\[E\circ B_{\mathbb{H}}(z_0;\frac{\eps_0}{8}) \subset \{(r,\theta)\in \H; \ e^{-\frac{\eps_0}{8}}\leq r \leq e^{\lsys(X)+\frac{\eps_0}{8}} \ \emph{and} \ \sin(\theta)\leq  \frac{16}{\lsys(X)}\}.\]
Set $$\mathcal{S}:=\{(r,\theta)\in \H; \ e^{-\frac{\eps_0}{8}}\leq r \leq e^{\lsys(X)+\frac{\eps_0}{8}} \ \emph{and} \ \sin(\theta)\leq  \frac{16}{\lsys(X)}\}.$$
By \eqref{ball-disj-1} we know that the balls $\{E\circ B_{\mathbb{H}}(z_0;\frac{\eps_0}{8})\}$ are pairwisely disjoint. Then it follows by \eqref{i-exp} and \eqref{lar-1-3} that
\begin{eqnarray}\label{u-2}
&&  H(z_0) \leq\sin^2(\theta(z_0))+ 4c(\frac{\eps_0}{8})\int_{\mathcal{S}} \sin^2{\theta}\dArea\\
 &&\leq1+ 8c(\frac{\eps_0}{8}) \int_{0}^{\arcsin(\min\{\frac{16}{\lsys(X)},1\}) }  \int_{e^{-\frac{\eps_0}{8}}}^{e^{\lsys(X)+\frac{\eps_0}{8}}}  \frac{\sin^2{\theta}}{r^2\sin^2{\theta}}rdrd\theta \nonumber\\
&&=1+ 8 c(\frac{\eps_0}{8}) \cdot (\lsys(X)+\frac{\eps_0}{4})\cdot \arcsin(\min\{\frac{16}{\lsys(X)},1\}) \nonumber \\
&&\leq 1+10 c(\frac{\eps_0}{8}) \cdot \lsys(X)\cdot \arcsin(\min\{\frac{16}{\lsys(X)},1\}) \nonumber
\end{eqnarray}

\noindent where we apply $\dArea=\frac{|dz|^2}{y^2}=\frac{rdrd\theta}{r^2\sin^2\theta}$ in the second inequality and $\lsys(X)\geq \eps_0>0$ in the last inequality. As $\lsys(X)\to \infty$, the right hand side above goes to $(1+160c(\frac{\eps_0}{8}))$ which is bounded. Recall that we always assume that $\lsys(X)\geq \eps_0>0$. Therefore there exists a uniform constant $C'_1(\eps_0)>0$ such that \[H(z_0)=\max_{z \in \mathbb{H}}H(z)\leq C'_1(\eps_0).\] 

The proof is complete.
\ep

\br
The proof above does not cover Proposition \ref{ub-lp} for the case that $\lsys(X)\to 0$ because the constant $c(\eps_0) \to \infty$ as $\eps_0\to 0$. We will use a different way to prove it in the next lemma.  
\er

Now we deal with the case that $\lsys(X)$ is short. For this case actually we will prove a more general result which does not require $\alpha$ to be a systolic curve of $X$. This is also a special case of the uniform upper bounds in Part (2) of Theorem \ref{s-ub-lp} for $p=\infty$. More precisely,
\bl \label{uub-small}
For any given constant $0<L_0<\frac{1}{1000}$ and $\alpha \subset X \in \sM_g$ with $\ell_{\alpha}(X) \leq L_0$, then there exists a uniform constant $C_2(L_0)>0$, only depending on $L_0$, such that
\[||\grad \ell_{\alpha}(X)||_\infty \leq C_2(L_0).\]
\el

\bp 
We follow the argument in \cite[Lemma 2.2]{Wolpert92} or \cite[Proposition 6]{Wolpert12}. By \eqref{2-3-0-0} it suffices to show that \[\max_{z \in \mathbb{H}}H(z)\leq C'_2(L_0)\]
for some uniform constant $C'_2(L_0)>0$ only depending on $L_0$.

Let $z_0\in \mathcal{F}_\alpha \subset \mathcal{A}_{\alpha}$, where $\mathcal{A}_{\alpha}=\{z\in \H; \ 1\leq |z|\leq e^{\ell_\alpha(X)}\}$ and $\mathcal{F}_\alpha\subset \mathcal{A}_\alpha$ is the fundamental domain which contains the lift $\{\textbf{i}\cdot t\}_{1\leq t \leq e^{\ell_\alpha(X)}}$ of the closed geodesic $\alpha$, such that
\[H(z_0)=\max_{z \in \mathbb{H}}H(z).\]
By Proposition \ref{p-1} we know that
\bear \label{larg-1-1}
\dist_{\H}(z_0, \textbf{i}\mathbb{R^+})\leq \ln{(1+\sqrt{2})}.
\eear

\noindent By Lemma \ref{collar-1} we know that the width $w(\alpha)$ of the maximal collar of $\alpha$ satisfies that
\bear 
w(\alpha)&\geq& \frac{1}{2} \ln \frac{\cosh{\frac{\ell_{\alpha}(X)}{2}}+1}{\cosh{\frac{\ell_{\alpha}(X)}{2}}-1} \\
&=& \ln{(\frac{e^{\frac{\ell_{\alpha}(X)}{2}}+1}{e^{\frac{\ell_{\alpha}(X)}{2}}-1})} \nonumber \\
&\geq & \ln{(\frac{2}{\ell_{\alpha}(X)})} \nonumber
\eear
where we apply $\ell_\alpha(X)<\frac{1}{1000}$ in the last inequality. In particular, 
\beqar\label{ln2000}
w(\alpha)>\ln{(2000)}.
\eeqar  
Let $\inj(\pi(z_0))$ be the injectivity radius of $\pi(z_0) \in X$. By \eqref{larg-1-1} we have
\[\dist(\pi(z_0), \alpha)\leq \ln{(1+\sqrt{2})}.\]
Then it follows by the triangle inequality that the geodesic ball $B(\pi(z_0);1)$ in $X$ centered at $\pi(z_0)$ of radius $1$ is contained in the maximal collar of $\alpha$. That is, 
\be
B(\pi(z_0);1) \subset \mathcal{C}_{w(\alpha)}(\alpha)
\ene
which together with the Collar Lemma \cite[Theorem 4.1.6]{Buser10} implies that 
\bear\label{collar-in}
\inj(\pi(z_0))\geq \frac{\ell_\alpha(X)}{2}.
\eear

\noindent Now we apply the mean value inequality. Recall that \eqref{H-expr} says that
\[H(z)=\sum_{E\in \left<A\right>\backslash \Gamma} \frac{1}{|E(z)|^2}\times \frac{1}{\rho(E(z))}=\sum_{E\in \left<A\right>\backslash \Gamma} \sin^2 (\theta(E(z))).\]
Then it follows by \eqref{i-exp} and Lemma \ref{mvp} that
\bear \label{smal-1-3}
H(z_0) &\leq& 4 c(1)\sum_{E \in \left<A\right>\backslash \Gamma} \int_{ B_{\mathbb{H}}(z_0;1)}\sin^2 (\theta(E(z)))\dArea\\
&=&4 c(1) \int_{B_{\mathbb{H}}(z_0;1)} H(z) \dArea. \nonumber 
\eear 

\noindent Similar as \cite[Lemma 2.2]{Wolpert92} or \cite[Proposition 6]{Wolpert12}, the projection $\pi:\bigcup_{\gamma \in \Gamma} \gamma \circ B_\H(z_0;1)$ onto its image in $X$ has multiplicity at most $\frac{c'}{\inj(\pi(z_0))}$ where $c'>0$ is a uniform constant. Indeed, for any $\gamma \notin \left<A\right>$ it follows by the triangle inequality that $\dist_H(\gamma \circ z_0,z_0)\geq \dist_H(\gamma \circ z_0, \textbf{i}\R^+)- \dist_H(z_0, \textbf{i}\R^+)\geq w(\alpha)-\ln{(1+\sqrt{2})}>\ln (2000)-1>2$. This tells that the multiplicity of the projection $\pi:\bigcup_{\gamma \in \Gamma} \gamma \circ B_\H(z_0;1)$ onto its image only comes from the cyclic group $\left<A\right>$. Which in particular implies that the multiplicity of the projection is at most $\frac{c'}{\inj(\pi(z_0))}$ for some uniform constant $c'>0$. 

\noindent Recall that $H(\gamma \circ z)=H(z)$. Thus, the last integral in \eqref{smal-1-3} is bounded above by
\[\frac{c'}{\inj(\pi(z_0))} \cdot \int_{\mathcal{A}_\alpha \bigcap (\bigcup_{E \in \left<A\right>/\Gamma}E \circ B_{\H}(z_0;1))} \sin ^2(\theta(z)) \dArea.\]
Therefore, we have
\bear
\ \ \ \ \ \ \ H(z_0) &\leq& \frac{4c' \cdot c(1)}{\inj(\pi(z_0))}\cdot \int_{\mathcal{A}_\alpha \bigcap (\bigcup_{E \in \left<A\right>/\Gamma}E \circ B_{\H}(z_0;1))} \sin^2 (\theta(z)) \dArea  \\
&\leq & \frac{4c' \cdot c(1)}{\inj(\pi(z_0))}\cdot \int_{\mathcal{A}_\alpha} \sin^2 (\theta(z)) \dArea \nonumber \\
&=&\frac{4c' \cdot c(1)}{\inj(\pi(z_0))}\cdot  \int_{0}^{\pi} \int_{1}^{e^{\ell_\alpha(X)}} \frac{\sin^2{\theta}}{r^2\sin^2{\theta}}rdrd\theta \nonumber\\
&=& \frac{4\pi c' \cdot c(1)}{\inj(\pi(z_0))}\cdot \ell_\alpha(X) \nonumber \\
&\leq & 8\pi c' \cdot c(1) \nonumber
\eear
where we apply \eqref{collar-in} in the last inequality. Then the conclusion follows by choosing
\[C'_2(L_0)=8\pi c' \cdot c(1).\]

The proof is complete.
\ep

\br
Let $\mu=\grad \ell_\alpha(X)$ in \cite[Definition 10]{Wolpert17}. Then the quantity $Comp(\grad \ell_\alpha(X))$ is comparable to $\frac{1}{|| \grad \ell_\alpha(X)||_{wp}}$. Thus, \cite[Lemma 11]{Wolpert17} implies that if $\alpha$ has length at most $c_0$, then there exists two positive constants $c'$ and $c''$ such that
\[c'\leq|| \grad \ell_\alpha(X)||_{\infty}\leq c''.\]
The upper bound $C_2(L_0)$ in Lemma \ref{uub-small} is uniform. And we will use this upper bound and \eqref{Rie-lb} of Riera to show that $|| \grad \ell_\alpha(X)||_{\infty}\geq C_2'(L_0)$ for some uniform constant $C_2'(L_0)>0$. The methods for obtaining these two bounds are similar.
\er

\br
The upper bound $\frac{1}{1000}$ in the assumption of Lemma \ref{uub-small} is clearly not optimal. However it is already good enough for proving the uniform upper bounds of Theorem \ref{s-ub-lp} for $p=\infty$.
\er

Now we prove Proposition \ref{uub-infi}.

\bp [Proof of Proposition \ref{uub-infi}]
If $\lsys(X)<\frac{1}{1000}$, it follows by Lemma \ref{uub-small} that for any $\alpha \subset X$ with $\ell_\alpha(X)=\lsys(X)$ we have
\[||\grad \ell_{\alpha}(X)||_\infty \leq C_2(\frac{1}{1000})\]
where $C_2(\cdot)$ is the constant in Lemma \ref{uub-small}.

If $\lsys(X)\geq \frac{1}{1000}$, it follows by Lemma \ref{uub-large} that
\[||\grad \ell_{\alpha}(X)||_\infty \leq C_1(\frac{1}{1000})\]
where $C_1(\cdot)$ is the constant in Lemma \ref{uub-large}.

Thus, we have that for any $X\in \sM_g$ and $\alpha \subset X$ with $\ell_{\alpha}(X)=\lsys(X)$,
\[||\grad \ell_{\alpha}(X)||_\infty\leq \max \{C_1(\frac{1}{1000}),C_2(\frac{1}{1000})\}.\]

Then the conclusion follows.
\ep

Next we prove the uniform upper bounds in Part (2) of Theorem \ref{s-ub-lp} for $p=\infty$. That is to extend the constant $\frac{1}{1000}$ in Lemma \ref{uub-small} to an arbitrary fixed positive constant. More precisely,
\begin{proposition} \label{uub-small-gene}
For any given constant $L_0>0$ and simple loop $\alpha \subset X \in \sM_g$ with $\ell_{\alpha}(X) \leq L_0$, then there exists a uniform constant $C_3(L_0)>0$, only depending on $L_0$, such that
\[||\grad \ell_{\alpha}(X)||_\infty \leq C_3(L_0).\]
That is,
\[||\grad \ell_{\alpha}(X)||_\infty \prec 1.\]
\end{proposition}
\bp
By Lemma \ref{uub-small} one may assume that $$\frac{1}{1000}\leq \ell_{\alpha}(X)\leq L_0$$ where $L_0\geq \frac{1}{1000}$. It follows by \eqref{2-3-0-0} and Proposition \ref{p-1} that
\bear
||\grad \ell_{\alpha}(X)||_\infty \leq \frac{2}{\pi}\cdot \max_{\pi(z) \in X;\ \dist(\pi(z), \alpha)\leq \ln(\sqrt{2}+1)} h(\pi(z))
\eear
where $h(\pi (z))=H(z)$ and $H(z)=\sum_{E\in \left<A\right>\backslash \Gamma}\sin^2 (\theta(E(z)))$ given in \eqref{H-expr}. 

\noindent \textsl{Claim:} there exists a uniform constant $s_0=s_0(L_0)>0$ only depending on $L_0$ such that 
\bear \label{lb-3000}
\inf_{p\in X; \ \dist(p, \alpha)\leq \ln{(\sqrt{2}+1)}}\inj(p)\geq s_0.
\eear
\textsl{Proof of Claim.} Let $p\in X$ with $\dist(p,\alpha)\leq \ln{(\sqrt{2}+1)}$ such that
\bear \label{3000}
\inj(p) <\frac{1}{2000};\eear
otherwise we are done. Since $\ell_\alpha(X)\leq L_0$, it follows by the triangle inequality that
\bear\label{L_0/2}
\alpha \subset B(p;\ln{(\sqrt{2}+1)}+\frac{L_0}{2}).
\eear

\noindent Since $\ell_{\alpha}(X)\geq \frac{1}{1000}$, by \eqref{3000} and the Collar Lemma \cite[Theorem 4.1.6]{Buser10} there exists a closed geodesic $\beta\neq \alpha$ such that
\[p \in \mathcal{C}_{w(\beta)}(\beta)\]
where $\mathcal{C}_{w(\beta)}(\beta)$ is the maximal collar of $\beta$. Let $\partial \mathcal{C}_{w(\beta)}(\beta)$ be the boundary of the maximal collar $\mathcal{C}_{w(\beta)}(\beta)$. Since $\beta$ is the unique closed geodesic in $\mathcal{C}_{w(\beta)}(\beta)$,  
\bear \label{dist-boun}
\dist(p, \partial \mathcal{C}_{w(\beta)}(\beta)) \leq  \ln{(\sqrt{2}+1)}+\frac{L_0}{2};
\eear
otherwise the geodesic ball $B(p;\ln{(\sqrt{2}+1)}+\frac{L_0}{2})\subset \mathcal{C}_{w(\beta)}(\beta)$ that together with \eqref{L_0/2} implies that \[\alpha \subset \mathcal{C}_{w(\beta)}(\beta)\] which is impossible because the collar $\mathcal{C}_{w(\beta)}(\beta)$ can not contain two different simple closed geodesics $\alpha$ and $\beta$. Recall that the Collar Lemma \cite[Theorem 4.1.6]{Buser10} tells that
\[\sinh{(\inj(p))}=\cosh {(\frac{\ell_\beta(X)}{2})} \cosh (\dist(p, \partial \mathcal{C}_{w(\beta)}(\beta)))-\sinh (\dist(p, \partial \mathcal{C}_{w(\beta)}(\beta)))\]
which together with \eqref{dist-boun} implies that
\bear
\ \ \ \ \ \ \ \ \ \ \sinh{(\inj(p))} &\geq& \cosh (\dist(p, \partial \mathcal{C}_{w(\beta)}(\beta)))-\sinh (\dist(p, \partial \mathcal{C}_{w(\beta)}(\beta))) \\
&=& e^{-\dist(p, \partial \mathcal{C}_{w(\beta)}(\beta))}   \nonumber \\
&\geq & (\sqrt{2}-1)\cdot e^{-\frac{L_0}{2}}. \nonumber
\eear
So we have
\[\inj(p)\geq \sinh^{-1}((\sqrt{2}-1)\cdot e^{-\frac{L_0}{2}}).\]
Since $p\in X$ with $\dist(p,\alpha)\leq \ln{(\sqrt{2}+1)}$ is arbitrary, the claim follows by setting
\bear
s_0=s_0(L_0)=\min\{\frac{1}{2000},\sinh^{-1}((\sqrt{2}-1)\cdot e^{-\frac{L_0}{2}})\}.
\eear

The remaining follows by a standard \emph{unfolding} argument. More precisely we let $\pi(z_0)\in X$ such that $$h(\pi(z_0))=\sup_{p\in X}h(p).$$ By Proposition \ref{p-1} and \eqref{lb-3000} we know that \bear \inj(\pi(z_0))\geq s_0.\eear Then it follows by the triangle inequality that for any $\gamma_1\neq \gamma_2 \in \Gamma$, the geodesic balls satisfy that
\[\gamma_1\circ B_\H(z_0; \frac{s_0}{4}) \cap \gamma_2\circ B_\H(z_0; \frac{s_0}{4})=\emptyset\]
where $z_0 \in \{(r,\theta)\in \H; \ 1\leq r \leq e^{\ell_\alpha(X)}\}$ is a lift of $\pi(z_0)$. Then it follows by \eqref{i-exp} and Lemma \ref{mvp} that
\beqar
h(\pi(z_0))&=&H(z_0)=\sum_{E\in \left<A\right>\backslash \Gamma}\sin^2 (\theta(E(z_0)))\\
&\leq& 4 c(\frac{s_0}{4}) \cdot \sum_{E\in \left<A\right>\backslash\Gamma}\int_{B_\H(E\circ z_0; \frac{s_0}{4})}\sin^2 (\theta(z)) \dArea.
\eeqar
These balls $\{E\circ B_{\mathbb{H}}(z_0;\frac{s_0}{4})\}$ are pairwisely disjoint and for all $E\notin \left<A\right>$,
\[E\circ B_{\mathbb{H}}(z_0;\frac{s_0}{4}) \subset \{(r,\theta)\in \H; \ e^{-\frac{s_0}{4}}\leq r \leq e^{\ell_\alpha(X)+\frac{s_0}{4}}\}.\]
Thus, we have
\bear
&& ||\grad \ell_{\alpha}(X)||_\infty \leq \frac{2}{\pi}\cdot h(\pi(z_0)) \\
&& \leq \frac{8}{\pi}\cdot c(\frac{s_0}{4}) \cdot \int_{z\in \H; \ e^{-\frac{s_0}{4}}\leq r \leq e^{\ell_\alpha(X)+\frac{s_0}{4}}}\sin^2(\theta(z))\dArea \nonumber \\
&& = \frac{8}{\pi}\cdot c(\frac{s_0}{4}) \cdot \int_0^{\pi} \int_{e^{-\frac{s_0}{4}}}^{e^{\ell_\alpha(X)+\frac{s_0}{4}}} \frac{\sin^2{\theta}}{r^2\sin^2{\theta}}rdrd\theta               \nonumber \\
&& =8 \cdot c(\frac{s_0}{4})\cdot (\frac{s_0}{2}+\ell_\alpha(X)) \nonumber \\
&& \leq 8 \cdot c(\frac{s_0}{4})\cdot (\frac{s_0}{2}+L_0)\nonumber
\eear
which completes the proof by setting $C_3(L_0)= 8 \cdot c(\frac{s_0}{4})\cdot (\frac{s_0}{2}+L_0)$.
\ep


\section{Uniform bounds for $L^p(1\leq p \leq \infty)$-norm}\label{s-lp-asymp}
Recall that we always use the notation: $r^{\frac{1}{\infty}}=1$ for any $r>0$. 

In this section we prove
\bt [=Theorem \ref{s-ub-lp}] \label{ub-lp} For any $p \in [1,\infty]$ and $X\in \sM_g$ $(g\geq 2)$, then we have 
\ben
\item for any $\alpha \subset X$ with $\ell_{\alpha}(X)=\lsys(X)$, 
\[||\grad \ell_{\alpha}(X)||_p \asymp \lsys(X)^{\frac{1}{p}}.\]

\item For any simple loop $\beta \subset X$ with $\ell_{\beta}(X) \leq L_0$ where $L_0>0$ is any given constant, 
\[||\grad \ell_{\beta}(X)||_p \asymp \ell_\beta(X)^{\frac{1}{p}}.\]
\een
\et

We will first show Theorem \ref{ub-lp} for the cases that $p=1,2$ and $\infty$. Then for general $p\in (1,2)\cup (2, \infty)$, it follows by a standard argument through using H\"older's inequality for integrals.

Before proving Theorem \ref{ub-lp}, we first show that $||\grad \ell_{\alpha}(X)||_1 \prec \ell_{\alpha}(X)$ where $\alpha\subset X$ is not necessary to be a systolic curve. More precisely,
\begin{lemma} \label{up-l1}
For any non-trivial loop $\alpha \subset X \in \sM_g$, then we have
\[||\grad \ell_{\alpha}(X)||_1 \leq 2 \ell_{\alpha}(X).\]
\end{lemma}
\bp
The proof is a standard \emph{unfolding} argument. Recall that $$\grad \ell_{\alpha}(X)(z)=\frac{2}{\pi}\frac{\overline{\Theta}_{\alpha}(z)}{\rho(z)|dz|^2}$$ where $\Theta_{\alpha}(z)=\displaystyle \sum_{E\in \left<A\right>\backslash\Gamma} \frac{E'(z)^2}{E(z)^2}dz^2$. Let $\mathbb{F}\subset \{z\in \mathbb{H}; 1\leq |z|\leq e^{\ell_{\alpha}(X)}\}$ be a fundamental domain of $X$. Then
\beqar
||\grad \ell_{\alpha}(X)||_1&=&\int_{\mathbb{F}} |\grad \ell_{\alpha}(X)(z)| \cdot  \rho(z)|dz|^2 \\
&\leq &  \frac{2}{\pi}\sum_{E\in \left<A\right>\backslash \Gamma} \int_{\mathbb{F}} \frac{|E'(z)|^2}{|E(z)|^2}|dz|^2 \\
&\leq & \frac{2}{\pi} \int_{z\in \H;\ 1\leq |z| \leq e^{\ell_{\alpha}(X)}}\frac{1}{|z|^2}|dz|^2 \\
&=& \frac{2}{\pi}  \int_0^\pi \int_1^{e^{\ell_\alpha(X)}}\frac{1}{r^2} \cdot rdr d\theta  \\
&=& 2 \ell_{\alpha}(X).
\eeqar 

The proof is complete.
\ep

Now we are ready to prove Theorem \ref{ub-lp}.

\bp [Proof of Theorem \ref{ub-lp}]
We first prove Part (1). The proof is splitted into the following four cases.

\emph{Case-1: $p=1$.} First by Lemma \ref{up-l1} we clearly have
\[||\grad \ell_{\alpha}(X)||_1 \prec \ell_{\alpha}(X).\]
For the other direction, since $||\grad \ell_{\alpha}(X)||_2^2\leq ||\grad \ell_{\alpha}(X)||_1 \cdot ||\grad \ell_{\alpha}(X)||_\infty $, by \eqref{Rie-lb} of Riera and Proposition \ref{uub-infi} we have
\beqar
||\grad \ell_{\alpha}(X)||_1 \geq \frac{||\grad \ell_{\alpha}(X)||_2^2}{ ||\grad \ell_{\alpha}(X)||_\infty}\succ \ell_{\alpha}(X).
\eeqar
Thus, we have
\bear \label{ub-lp-f-1}
||\grad \ell_{\alpha}(X)||_1 \asymp  \lsys(X).
\eear
\

\emph{Case-2: $p=2$.} First \eqref{Rie-lb} of Riera says that
\[||\grad \ell_{\alpha}(X)||_2 \succ (\ell_{\alpha}(X))^{\frac{1}{2}}.\]
For the other direction, since $||\grad \ell_{\alpha}(X)||_2^2\leq ||\grad \ell_{\alpha}(X)||_1 \cdot ||\grad \ell_{\alpha}(X)||_\infty $, by Proposition \ref{uub-infi} and Lemma \ref{up-l1} we have
\[||\grad \ell_{\alpha}(X)||_2 \leq \sqrt{||\grad \ell_{\alpha}(X)||_1 \cdot ||\grad \ell_{\alpha}(X)||_\infty } \prec (\ell_{\alpha}(X))^{\frac{1}{2}}.\]
Thus, we have 
\bear \label{ub-lp-f-2}
||\grad \ell_{\alpha}(X)||_2 \asymp  \lsys(X)^{\frac{1}{2}}.
\eear
\

\emph{Case-3: $p=\infty$.} First Proposition \ref{uub-infi} says that
\[||\grad \ell_{\alpha}(X)||_\infty \prec 1.\]
For the other direction, since $||\grad \ell_{\alpha}(X)||_2^2\leq ||\grad \ell_{\alpha}(X)||_1 \cdot ||\grad \ell_{\alpha}(X)||_\infty $, by Case-1 and Case-2 we have
\[||\grad \ell_{\alpha}(X)||_\infty \geq \frac{||\grad \ell_{\alpha}(X)||_2^2}{||\grad \ell_{\alpha}(X)||_1 } \asymp 1.\]
Thus, we have
\bear \label{ub-lp-f-3}
||\grad \ell_{\alpha}(X)||_\infty \asymp  1.
\eear
\

\emph{Case-4: general $p\in (1,2)\cup(2,\infty)$.} First since $$||\grad \ell_{\alpha}(X)||_p^p\leq ||\grad \ell_{\alpha}(X)||_1 \cdot ||\grad \ell_{\alpha}(X)||_\infty^{p-1},$$ by Proposition \ref{uub-infi} and Case-1 we have
\[||\grad \ell_{\alpha}(X)||_p \prec \ell_{\alpha}(X)^{\frac{1}{p}}.\]
For the other direction, let $q\in (1,2)\cup(2,\infty)$ such that $\frac{1}{p}+\frac{1}{q}=1$. Similar as above we have
\bear \label{ub-lp-f-4-0}
||\grad \ell_{\alpha}(X)||_q \prec \ell_{\alpha}(X)^{\frac{1}{q}}.
\eear

\noindent Recall that H\"older inequality says that
\bear
||\grad \ell_{\alpha}(X)||_p \geq \frac{||\grad \ell_{\alpha}(X)||_2^2}{||\grad \ell_{\alpha}(X)||_q}. \nonumber
\eear
By Case-2 and \eqref{ub-lp-f-4-0} we have
\bear
||\grad \ell_{\alpha}(X)||_p \succ \ell_{\alpha}(X)^{1-\frac{1}{q}}=\ell_{\alpha}(X)^{\frac{1}{p}}.\nonumber
\eear
Thus, we have
\bear \label{ub-lp-f-4}
||\grad \ell_{\alpha}(X)||_p \asymp \lsys(X)^{\frac{1}{p}}.
\eear
Then the conclusion follows by \eqref{ub-lp-f-1}, \eqref{ub-lp-f-2}, \eqref{ub-lp-f-3} and \eqref{ub-lp-f-4}.\\

To prove Part (2), it follows by the same argument as the proof of Part (1) by applying Proposition \ref{uub-small-gene} instead of Proposition \ref{uub-infi}.
\ep

\section{Applications to \wep geometry}\label{s-appl-wp}
In this section we make several applications of Theorem \ref{s-ub-lp} to the \wep geometry of $\sM_g$.

\subsection{A uniform lower bound for \wep distance} As discussed in the introduction, Theorem \ref{ub-lp} can be applied to prove the following two results on the global \wep geometry of $\sM_g$.
\bt \cite[Theorem 1.3]{W-inradius}\label{wp-lb}
For all $X,Y\in \Teich(S_{g})$,
\[|\sqrt{\ell_{sys}(X)}-\sqrt{\ell_{sys}(Y)}|\prec \dist_{wp}(X,Y)\]
where $\dist_{wp}$ is the \wep distance.
\et

\bt \cite[Theorem 1.1]{W-inradius}
For all $g\geq 2$, $$\InR(\sM_g)\asymp \sqrt{\ln{(g)}}.$$
\et

\subsection{Uniform bounds on \wep curvatures} In this subsection we prove several new uniform bounds on \wep curvatures. Recall the formula \eqref{HolK-eq} says that for all $\mu \in T_X \sM_g$ with $||\mu||_2=1$, 
\begin{equation}\label{HolK-eq-1}
\HolK(\mu) \asymp -\int_{X} |\mu|^4 \dArea. 
\end{equation} 

Our first bound on \wep curvature is as follows.
\begin{theorem}[=Theorem \ref{s-ub-holk}]\label{ub-holk}
For any $X \in \sM_g$, let $\alpha \subset X$ be a simple closed geodesic satisfying 
\ben
\item either $\ell_{\alpha}(X)=\lsys(X)$ or

\item $\ell_\alpha(X)\leq L_0$ where $L_0>0$ is any fixed constant,
\een
then the \wep holomorphic sectional curvature
$$\HolK(\grad \ell_{\alpha})(X) \asymp \frac{-1}{\ell_{sys}(X)}.$$
\end{theorem}

\begin{proof}[Proof of Theorem \ref{ub-holk}]
By \eqref{HolK-eq-1} we have
\bear
\HolK(\grad \ell_{\alpha})(X)  \asymp -\frac{\int_{X}|\grad \ell_{\alpha}(X)|^4\dArea}{(\int_{X}|\grad \ell_{\alpha}(X)|^2\dArea)^2}.
\eear
We apply Theorem \ref{ub-lp} for the cases that $p=2,4$ to obtain
\bear
\int_{X}|\grad \ell_{\alpha}(X)|^4\dArea \asymp \lsys(X)
\eear
and
\bear
(\int_{X}|\grad \ell_{\alpha}(X)|^2\dArea)^2 \asymp \lsys^2(X).
\eear

Then the conclusion follows by these three equations above.
\end{proof}

Buser-Sarnak in \cite{BS94} show that $\max_{X\in \sM_g}\ell_{sys}(X) \asymp \ln (g)$ for all $g\geq 2$. A direct consequence of Theorem \ref{ub-holk} is as follows.
\begin{corollary}For all $g\geq 2$, 
$$\sup_{X \in \sM_g}\HolK(\grad \ell_{\alpha})(X) \asymp \frac{-1}{\ln (g)}$$
where $\alpha \subset X$ satisfies that $\ell_{\alpha}(X)=\ell_{sys}(X)$.
\end{corollary}

Another application of Theorem \ref{ub-lp} is to show that the minimum \wep holomorphic sectional curvature at any $X\in \sM_g$ is bounded from above by a uniform negative number. More precisely,
\bt [=Theorem \ref{s-holk-uup}]\label{holk-uup}
For any $X \in \sM_g$, 
$$\min_{\mu\in T_{X}\sM_g}\HolK(\mu) \prec -1<0.$$
\et

\bp
Let $X \in  \sM_g$ be arbitrary. We split the proof into two cases.

Case-1: $\ell_{sys}(X)=\ell_{\alpha}(X)<100$ for some $\alpha \subset X$. It follows by Theorem \ref{ub-holk} that 
\bear
\min_{\mu\in T_{X}\sM_g}\HolK(\mu) &\leq& \HolK(\grad \ell_{\alpha})(X) \\
&\asymp& \frac{-1}{\ell_{sys}(X)} \nonumber \\
&\prec&  -1 \nonumber
\eear
where we apply $\ell_{sys}(X)<100$ in the last equation.

Case-2:  $\ell_{sys}(X)\geq 100$. We use the $\mu_0(z)=\sum_{\gamma \in \Gamma} \frac{\overline{r'(z)^2}}{\rho(z)}\frac{d \overline{z}}{dz} \in T_X \sM_g$ instead of $\grad \ell_{\alpha}(X)$. Since $\lsys(X)\geq 100> 2\ln{(3+2\sqrt{2})}$, it follows by \cite[Theorem 6.1]{W17-wpc} or the proof of \cite[Theorem 1.1]{Wolf-W-1} that
\beqar \HolK(\mu_0)(X)\prec -1.\eeqar

\noindent Thus, we have if $\ell_{sys}(X)\geq 100$,
\bear
\min_{\mu\in T_{X}\sM_g}\HolK(P) \leq \HolK(\mu_0)(X)\prec -1.
\eear

Therefore the conclusion follows by the two cases above.
\ep

\br
Let $\tilde{Q}: \wedge^2T_X\sM_g \to \wedge^2T_X\sM_g$ be the real Riemannian curvature operator. This is an endomorphism of a $(3g-3)(6g-7)$ dimensional vector space. It was shown in \cite{W14-co} that $\tilde{Q}$ is non-positive definite. Moreover, it was shown in \cite[Theorem 1.2]{Wolf-W-1} that the $\ell^\infty$-norm $||\tilde{Q}||_{\ell^{\infty}}(X)$ of $\tilde{Q}$ at any $X\in \sM_g$ satisfies that \[||\tilde{Q}||_{\ell^{\infty}}(X) \geq \frac{1}{2\pi}.\] 
In particular, we have $||\tilde{Q}||_{\ell^{\infty}}(X)\succ 1$. The proof of this inequality highly depends on the \wep curvature operator formula developed in \cite{W14-co}. By definition and the negativity of \wep curvature one knows that $$||\tilde{Q}||_{\ell^{\infty}}(X) \geq \max_{\mu \in T_{X} \sM_g}|\HolK(\mu)|=-\min_{\mu \in T_{X} \sM_g}\HolK(\mu).$$ 
Thus, by Theorem \ref{holk-uup} we also get $$||\tilde{Q}||_{\ell^{\infty}}(X)\succ 1.$$
Although the uniform lower bound is not as explicit as $\frac{1}{2\pi}$ in \cite[Theorem 1.2]{Wolf-W-1}, as a direct consequence of Theorem \ref{holk-uup}, we give a completely different proof on this uniform lower bound $||\tilde{Q}||_{\ell^{\infty}}(X)\succ 1$ without using the \wep curvature operator formula.
\er

\section{Applications to the $L^p$ metric on $\sM_g$}\label{Lp-metric-inco}
In this section we study the $L^p$ $(1< p \leq \infty)$ metric on $\sM_g$ and its relation to Theorem \ref{s-ub-lp}.

Let $(X,\sigma(z)|dz|^2)\in \sM_g$ be a hyperbolic surface and $\phi \in Q(X)$ be a holomorphic quadratic differential on $X$. For any $p\in [1,\infty]$, we define
\bear \label{hqd-lp}
||\phi||_{L^p}:=||\frac{\overline{\phi}}{\sigma}||_p=\left(\int_X \left(\frac{|\phi(z)|}{\sigma(z)}\right)^p \cdot \dArea \right)^{\frac{1}{p}}.
\eear

The $L^p$ metric on $\sM_g$ is defined by \eqref{hqd-lp} via duality. More precisely,
\begin{definition}
For any $p\in [1,\infty]$, the \emph{$L^p$-metric} on $\sM_g$ is defined as follows: for any $\mu \in T_X \sM_g$ which is a harmonic Beltrami differential on $X$ we define
\bear\label{def-Lp}
||\mu||_{L^p}:=\sup_{\phi \in Q(X); \ ||\phi||_{L^p}=1}\Re \int_X \left(\frac{\phi}{\sigma}\cdot \mu \right)\cdot \dArea
\eear
We call that $(\sM_g, ||\cdot||_{L^p})$ is \emph{the moduli space endowed with the $L^p$-metric}, and denote by $\dist_p(\cdot, \cdot)$ the distance function on $(\sM_g, ||\cdot||_{L^p})$.    
\end{definition}

\noindent By definition we know that $(\sM_g, ||\cdot||_{L^1})$ is the \tec metric on $\sM_g$ which is a complete Finsler metric. And 
$(\sM_g, ||\cdot||_{L^2})$ is the \wep metric on $\sM_g$ which is an incomplete K\"ahler metric.

Let $q\in [1,\infty]$ be the conjugate number of $p$, i.e., $\frac{1}{q}+\frac{1}{p}=1$. For \eqref{def-Lp} first by the H\"older inequality we have
\bear \label{7-up-canc}
||\mu||_{L^p}\leq ||\mu||_{q}
\eear
where $||\mu||_q$ is defined in \eqref{def-mu-p}. On the other hand, by choosing $\phi=\frac{\overline{\mu} \cdot \sigma}{||\mu||_{p}}$ in \eqref{def-Lp} we have
\bear\label{7-lo-canc}
||\mu||_{L^p}\geq \frac{||\mu||_{2}^2}{||\mu||_{p}}= \frac{||\mu||_{\WP}^2}{||\mu||_p}.
\eear
In particular, $$||\mu||_{L^2}=||\mu||_2=||\mu||_{\WP}.$$

Now we are ready to state our result in this section.
\bt[=Theorem \ref{i-Lp-incomplete}]\label{Lp-incomplete}
Let $X \in \sM_g$ and $\alpha \subset X$ be a non-trivial simple loop with $\ell_\alpha(X)\leq L_0$ where $L_0>0$ is any given constant. Then for any $p\in (1, \infty]$,
\[\dist_p(X, \sM_g^{\alpha})\prec (\ell_{\alpha}(X))^{1-\frac{1}{p}}\]
where $\sM_g^{\alpha}$ is the stratum of $\sM_g$ whose pinching curve is $\alpha$. In particular, the space $(\sM_g, ||\cdot||_{L^p})$ $(1< p \leq \infty)$ is incomplete.
\et

\bp
Let $\grad \ell_\alpha(X)$ be the \wep gradient of the geodesic length function $\ell_\alpha(\cdot)$ at $X$. Now we consider the integral curve of the vector field $-\frac{\grad \ell_\alpha}{||\grad \ell_\alpha||_{L^p}}$. More precisely, let $c:[0, s) \to \sM_g$ be a curve, where $s>0$ is length of the maximal interval, satisfying
\ben
\item $c(0)=X$ and

\item $c'(t)=-\frac{\grad \ell_\alpha(c(t))}{||\grad \ell_\alpha(c(t))||_{L^p}}$.
\een 
Direct computations show that $t$ is an arc-length parameter of $c(\cdot)$ in $(\sM_g, ||\cdot||_{L^p})$ and $\ell_\alpha(c(t_1))<\ell_\alpha(c(t_2))$ for any $s> t_1> t_2 \geq 0$. Thus, as $t\to s$, $c(t)$ goes to $\sM_g^\alpha$. So we have
\[\lim_{t\to s}\ell_\alpha(c(t))=0.\]
Since $\ell_\alpha(c(t))$ is decreasing, $\ell_\alpha(c(t))\leq L_0$ for all $t\in [0,s)$. It follows by Part (2) of Theorem \ref{s-ub-lp} that for all $t\in [0,s)$ and $r\in [1,\infty]$,
\bear \label{lp-c(t)}
||\grad \ell_\alpha(c(t))||_r\asymp (\ell_\alpha(c(t)))^{\frac{1}{r}}
\eear
which together with \eqref{7-up-canc} and \eqref{7-lo-canc} implies that
\bear \label{7-Lp-1/p}
|\grad \ell_\alpha(c(t))||_{L^p}\asymp (\ell_\alpha(c(t)))^{\frac{1}{q}}
\eear
where $\frac{1}{q}+\frac{1}{p}=1$. Thus, we have for all $s> t_1> t_2 \geq 0$,
\beqar
|(\ell_\alpha(c(t_1)))^{\frac{1}{q}}-(\ell_\alpha(c(t_2)))^{\frac{1}{q}}|&=&|\int_{t_2}^{t_1}\frac{(\ell_\alpha(c(t)))^{\frac{1}{q}-1}}{q} \left<\grad \ell_\alpha(c(t)),c'(t) \right>_{wp}dt|\\
&=& \frac{1}{q} \int_{t_2}^{t_1}   \frac{||\grad \ell_\alpha(c(t))||^2_{\WP}}{(\ell_\alpha(c(t)))^{\frac{1}{p}}\cdot |\grad \ell_\alpha(c(t))||_{L^p}}dt\\
&\asymp& \int_{t_2}^{t_1} \frac{\ell_\alpha(c(t))}{(\ell_\alpha(c(t)))^{\frac{1}{p}}\cdot (\ell_\alpha(c(t)))^{\frac{1}{q}}}dt \\
&=& t_1-t_2 \\
&\geq& \dist_p(c(t_1),c(t_2))
\eeqar
where in the last inequality we apply $t$ is an arc-length parameter for $c(\cdot)$. We choose $t_2=0$ and let $t_1 \to s$ to get
\bear
\dist_p(X,  \sM_g^{\alpha})\leq \liminf_{t\to s}\dist_p(c(t),c(0))\prec (\ell_\alpha(X))^{\frac{1}{q}}
\eear
which completes the proof because $\frac{1}{q}=1-\frac{1}{p}$.
\ep

\bibliographystyle{amsalpha}
\bibliography{ref}

\end{document}